\documentclass[11pt,oneside]{article}
\usepackage{amsmath,amssymb,amsthm,graphicx}

\newcommand{\prn}[1]{\left(#1\right)}
\newcommand{\pd}[2]{\frac{\partial#1}{\partial#2}}
\renewcommand{\vec}[1]{\mbox{\boldmath$#1$}}
\theoremstyle{plain}
\newtheorem{thm}{Theorem}
\newtheorem{cor}{Corollary}
\theoremstyle{definition}
\newtheorem{defn}{Definition}
\newtheorem{ex}{Example}
\theoremstyle{remark}
\newtheorem{rem}{Remark}

\begin{document}

\title{Minimal Positive Stencils in Meshfree Finite Difference Methods
for the Poisson Equation}
\author{
Benjamin Seibold \\
Department of Mathematics \\
Massachusetts Institute of Technology \\
77 Massachusetts Avenue \\
Cambridge MA 02139, USA}
\maketitle

%=============================================================================================
\begin{abstract}
Meshfree finite difference methods for the Poisson equation approximate the Laplace
operator on a point cloud. Desirable are positive stencils, i.e.~all neighbor entries
are of the same sign. Classical least squares approaches yield large stencils that are
in general not positive. We present an approach that yields stencils of minimal size,
which are positive. We provide conditions on the point cloud geometry, so that positive
stencils always exist. The new discretization method is compared to least squares
approaches in terms of accuracy and computational performance.
\end{abstract}
%=============================================================================================

%=============================================================================================
\section{Introduction}
%=============================================================================================
The numerical approximation of the Poisson equation is a fundamental task encountered
in many applications. Often it appears as a subproblem in a more complex computation,
for instance as a projection step in the simulation of incompressible
flows \cite{Chorin1968}. Finite difference methods approximate the equation on a finite
number of points. If the points can be placed on a regular grid, the approximation is
simple and yields symmetric matrices. However, in many cases a regular point
distribution is not possible or desired. Examples are the explicit representation of
complex geometries, or the point positions may be given by the application, for instance
from scattered measurements or in particle methods \cite{KuhnertTiwariMeshfree2002}.
If the points are distributed irregularly, neighborhood relations have to be
established. This could be done by constructing a mesh. However, meshing can be
costly, and thus may not be desired in applications with time-dependent geometries.
Instead, meshfree neighborhood criteria can be defined, and meshfree finite difference
stencils constructed. Consistency conditions for stencils are derived in
Sect.~\ref{seibold:sec:fd_poisson}. We are interested in approaches that yield
M-matrices, as explained in Sect.~\ref{seibold:sec:m_matrix_structure}.
This requires the stencils to be positive. Least squares approaches, outlined
in Sect.~\ref{seibold:sec:least_squares_approaches}, in general fail to yield
positive stencils.
In Sect.~\ref{seibold:sec:lp_approach} we present a new approach, based on sign
constrained linear minimization, that yields positive stencils.
In Sect.~\ref{seibold:sec:lp_conditions_positive_stencil} we provide conditions
on the point cloud geometry, so that positive stencils are guaranteed to exist.
Further conditions, derived in Sect.~\ref{seibold:sec:lp_matrix_connectivity},
ensure an M-matrix structure.
In Sect.~\ref{seibold:sec:numerics} the new approach is compared to classical methods
by numerical experiments.

%=============================================================================================
\section{Meshfree Finite Differences for the Poisson Equation}
\label{seibold:sec:fd_poisson}
%=============================================================================================
Consider the Poisson equation to be solved inside a domain
$\Omega\subset\mathbb{R}^d$
\begin{equation}
\begin{cases}
-\Delta u = f &\mathrm{in~}\Omega \\
        u = g &\mathrm{on~}\Gamma_D \\
        \frac{\partial u}{\partial n} = h &\mathrm{on~}\Gamma_N
\end{cases}
\label{seibold:eq:poisson_equation}
\end{equation}
where $\Gamma_D\cup\Gamma_N=\partial\Omega$.
Let a point cloud $X=\{\vec{x}_1,\dots,\vec{x}_n\}\subset\overline{\Omega}$ be given,
which consists of interior points $X_i\subset\Omega$ and boundary points
$X_b\subset\partial\Omega$. The point cloud is meshfree, i.e.~no information about
connection of points is provided. Meshfree finite difference approaches convert
problem \eqref{seibold:eq:poisson_equation} into a linear system
\begin{equation}
A\cdot\vec{\hat u} = \vec{\hat f}\;,
\label{seibold:eq:poisson_system}
\end{equation}
where the vector $\vec{\hat u}$ contains approximations to the values $u(\vec{x}_i)$.
The $i$-th row of the matrix $A$ consists of the \emph{stencil} corresponding to the
point $\vec{x}_i$. We assume that \eqref{seibold:eq:poisson_equation} admits a unique
smooth solution.

%---------------------------------------------------------------------------------------------
\subsection{Consistent Derivative Approximation}
\label{seibold:subsec:consistent_derivatives}
%---------------------------------------------------------------------------------------------
Consider a function $u\in C^2(\Omega\subset\mathbb{R}^d,\mathbb{R})$. We wish to
approximate $\Delta u(\vec{x}_0)$ using the function values of a finite number
of points in a circular neighborhood
$(\vec{x}_0,\vec{x}_1,\dots,\vec{x}_m)\in B(\vec{x}_0,r)$, where
$B(\vec{x}_0,r) = \{\vec{x}\in\overline{\Omega}:\|\vec{x}-\vec{x}_0\|<r\}$.
Define the distance vectors $\vec{\bar x}_i = \vec{x}_i-\vec{x}_0 \ \forall i=0,\dots,m$.
The function value at each neighboring point $u(\vec{x}_i)$ can be expressed by
a Taylor expansion
\begin{equation*}
u(\vec{x}_i) = u(\vec{x}_0)+\nabla u(\vec{x}_0)\cdot\vec{\bar x}_i
+\tfrac{1}{2}\nabla^2 u(\vec{x}_0):\prn{\vec{\bar x}_i\cdot\vec{\bar x}_i^T}+e_i\;.
\end{equation*}
We use the matrix scalar product $A:B=\sum_{i,j}A_{ij}B_{ij}$. The error in the
expansion is of order $e_i = O(r^3)$. A linear combination with coefficients
$(s_0,\dots,s_m)$ equals
\begin{align*}
\sum_{i=0}^m s_i u(\vec{x}_i) =
u(\vec{x}_0)\prn{\sum_{i=0}^m s_i}
+\nabla u(\vec{x}_0)\cdot\prn{\sum_{i=1}^m s_i\vec{\bar x}_i}\;\; \\
+\nabla^2 u(\vec{x}_0):\prn{\frac{1}{2}\sum_{i=1}^m
s_i\prn{\vec{\bar x}_i\cdot\vec{\bar x}_i^T}}
+\prn{\sum_{i=1}^m s_ie_i}\;.
\end{align*}
This approximates the Laplacian,
i.e.~$\sum_{i=0}^m s_i u(\vec{x}_i) = \Delta u(\vec{x}_0)+O(r^3)$,
if exactness for constant, linear and quadratic functions is satisfied
\begin{equation}
\sum_{i=0}^m s_i = 0 \quad , \quad
\sum_{i=1}^m\vec{\bar x}_i s_i = 0 \quad , \quad
\sum_{i=1}^m\prn{\vec{\bar x}_i\cdot\vec{\bar x}_i^T}s_i = 2I\;.
\label{seibold:eq:constraints_laplace}
\end{equation}
\begin{defn}
\label{seibold:def:consistency}
A stencil $(s_0,\dots,s_m)$ to a set of points
$(\vec{x}_0,\vec{x}_1,\dots,\vec{x}_m)\in B(\vec{x}_0,r)$ is called \emph{consistent}
(with the Laplace operator), if the constraints \eqref{seibold:eq:constraints_laplace}
are satisfied.
\end{defn}
The linear and quadratic constraints in \eqref{seibold:eq:constraints_laplace} can be
formulated as a linear system of equations
\begin{equation}
V\cdot\vec{s} = \vec{b}\;,
\label{seibold:eq:linear_system}
\end{equation}
where $V\in\mathbb{R}^{k\times m}$ is the Vandermonde matrix given
by $\vec{\bar x}_1,\dots,\vec{\bar x}_m$, and $\vec{s}\in\mathbb{R}^m$ is the
stencil vector. In 2d, with $\vec{\bar x}_i=(\bar x_i,\bar y_i)$, the system reads as
\begin{equation}
V = \prn{\begin{array}{ccc}
\bar x_1 & \dots & \bar x_m \\
\bar y_1 & \dots & \bar y_m \\
\bar x_1\bar y_1 & \dots & \bar x_m\bar y_m \\
\bar x_1^2 & \dots & \bar x_m^2 \\
\bar y_1^2 & \dots & \bar y_m^2
\end{array}}
\ , \quad
\vec{b} = \prn{\begin{array}{c}
0 \\ 0 \\ 0 \\ 2 \\ 2
\end{array}}\;.
\label{seibold:eq:der_consistent_system_2d}
\end{equation}
The number of constraints is $k = \frac{d(d+3)}{2}$.
The constant constraint in \eqref{seibold:eq:constraints_laplace} yields
$s_0 = -\sum_{i=1}^m s_i$.
Neumann boundary points can be treated in a similar manner.
Approximating $\pd{u}{\vec{n}}(\vec{x}_0)$ by $\sum_{i=0}^m s_i u(\vec{x}_i)$
leads to the constraints
\begin{equation}
\sum_{i=0}^m s_i = 0 \quad , \quad
\sum_{i=1}^m\vec{\bar x}_i s_i = \vec{n}\;.
\label{seibold:eq:constraints_Neumann}
\end{equation}
For each point, a meshfree finite difference approximation consists of two steps:
First, define which points are its neighbors. Typically, more neighbors than
constraints are chosen. Second, select a stencil.
If \eqref{seibold:eq:constraints_laplace} is underdetermined, a minimization problem
is formulated to select a unique stencil.

A set of neighbors around a central point is called in \emph{general configuration},
if the Vandermonde matrix $V$ has full rank. For $m$ neighboring points, one has no
solution if $m<k$, infinitely many solutions if $m>k$, and one solution if $m=k$.
In this case, the stencil $\vec{s}=V^{-1}\cdot\vec{b}$ can be computed by
elimination, or by formulas for the determinant of a multivariate Vandermonde
matrix \cite{Lorentz1992}.
If the points are not in general configuration, e.g.~for regular grids, the above
rules may fail. A solution can exist for $m<k$ (e.g.~5-point stencil in 2d), and no
solution may exist for $m>k$ (see example in \cite[p.~59]{SeiboldDiss2006}).
The concept of general configuration is unhandy and too strict. Solutions may be
acceptable, even if $V$ does not have full rank. The geometric condition presented
in Sect.~\ref{seibold:sec:lp_conditions_positive_stencil} ensures the existence of
stencils.

%---------------------------------------------------------------------------------------------
\subsection{Minimal and Positive Stencils}
\label{seibold:subsec:min_pos_stencils}
%---------------------------------------------------------------------------------------------
\begin{defn}
\label{seibold:def:minimal_stencil}
A consistent stencil $(s_0,\dots,s_m)$ is called \emph{minimal}, if $m\le k$.
\end{defn}
Minimal stencils are beneficial for the sparsity of the system matrix, resulting
in a lower memory consumption and a faster solution of system
\eqref{seibold:eq:poisson_system}. The total number of neighboring points is
proportional to the effort of applying the matrix to a vector, which is proportional
to the time for one step of a (semi-)iterative linear solver.
Of course, the few neighbors have to be chosen wisely, to preserve good convergence
rates of iterative solvers. The results in \cite{SeiboldMeshfree2007,SeiboldECMI2008}
indicate that this is the case with the presented approach.
\begin{rem}
\label{seibold:rem:no_continuous_dependence}
For minimal stencils, it is impossible that the stencil values depend continuously
on the point positions. Consider six points around a central point (in 2d), five of
which are selected neighbors. Consider a continuous movement of one of the neighbors
and the sixth point, such that at the end these two points have swapped their
positions, without them ever being in the same place. At some instance during this
movement, the sixth point has to become a neighbor, resulting in a jump in the
stencil values.
\end{rem}
\begin{rem}
If used in a particle method, the lack of smoothness in minimal stencils
(Rem.~\ref{seibold:rem:no_continuous_dependence}) may lead to a non-conservative
scheme. In an isolated Poisson solver and in particle methods that are not
conservative by construction (such as the finite pointset method
\cite{KuhnertTiwariMeshfree2002}) the advantage of optimal sparsity often
outweighs this drawback.
\end{rem}
\begin{defn}
\label{seibold:def:positive_stencil}
A consistent stencil $(s_0,\dots,s_m)$ is called \emph{positive}, if \linebreak
$s_1,\dots,s_m\ge 0$. Due to \eqref{seibold:eq:constraints_laplace} and
\eqref{seibold:eq:constraints_Neumann}, this implies for the central point $s_0<0$.
\end{defn}
Positive stencils yields the system matrix in \eqref{seibold:eq:poisson_system} to be
an L-matrix (Def.~\ref{seibold:def:L_matrix}), which gives rise to an M-matrix
structure (see Sect.~\ref{seibold:sec:m_matrix_structure}).
The desirability of positive stencils has been pointed out
by Demkowicz, Karafiat and Liszka \cite{DemkowiczKarafiatLiszka1984}.
Classical approaches do in general not yield positive stencils.
An ``optimal star selection'' \cite{DuarteLiszkaTworzyako1996} makes positive stencils
likely, but they are not guaranteed
(see Fig.~\ref{seibold:fig_neighborhood_four_quadrant}).
F{\"u}rst and Sonar derive topological conditions on point clouds for positive least
squares stencils in 1d \cite{FuerstSonar2001}.
In Sect.~\ref{seibold:sec:lp_approach} we present a strategy that approximates the
Poisson equation \eqref{seibold:eq:poisson_equation} on a point cloud by minimal
positive stencils.
In Sect.~\ref{seibold:sec:lp_conditions_positive_stencil} conditions on a point
cloud are presented (in 2d and 3d) which guarantee the existence of positive stencils.

%=============================================================================================
\section{M-Matrices}
\label{seibold:sec:m_matrix_structure}
%=============================================================================================
Meshfree finite difference matrices are in general non-symmetric. Consider two points
$\vec{x}_i$ and $\vec{x}_j$, each being a neighbor of the other, and a third point
$\vec{x}_k$ which is a neighbor of $\vec{x}_i$ but not a neighbor of $\vec{x}_j$.
Since each stencil entry depends on \emph{all} its neighbors, the point $\vec{x}_k$
influences the matrix entry $a_{ij}$, but not the matrix entry $a_{ji}$.

The negative Laplace operator in \eqref{seibold:eq:poisson_equation} is positive
definite. For non-symmetric matrices, we have to ask for slightly less than positive
definiteness. A property which implies a maximum principle and the convergence of
linear solvers, is the M-matrix structure.
\begin{defn}
\label{seibold:def:L_matrix}
A square matrix $A=(a_{ij})_{ij}\in\mathbb{R}^{n\times n}$ is called \emph{Z-matrix},
if $a_{ij}\le 0 \ \forall i\neq j$.
A Z-matrix is called \emph{L-matrix}, if $a_{ii}>0 \ \forall i$.
\end{defn}
We write $A\ge 0$ for $a_{ij}\ge 0 \ \forall i,j$. The same notation applies to vectors.
\begin{defn}
A regular matrix $A$ is called \emph{inverse positive}, if $A^{-1}\ge 0$.
\end{defn}
\begin{defn}
A Z-matrix is called \emph{M-matrix}, if it is inverse positive.
\end{defn}
We use the M-matrix property, since it yields a sufficient condition for inverse
positivity. There are inverse positive matrices that are not M-matrices, so another
approach would be to employ alternative characterizations of inverse positive
matrices \cite{FujimotoRanade2004}.

%---------------------------------------------------------------------------------------------
\subsection{Benefits of an M-Matrix Structure}
%---------------------------------------------------------------------------------------------
The Poisson equation satisfies maximum principles. For instance, consider
\eqref{seibold:eq:poisson_equation} with Dirichlet boundary conditions only.
If $f\le 0$ and $g\le 0$, then the solution satisfies $u\le 0$ \cite{Evans1998}.
A discretization by an M-matrix mimics this property in a discrete maximum
principle.
\begin{thm}
\label{thm:mmatrix_discrete_max_principle}
Let $A$ be an M-matrix. Then $A\vec{x}\le 0$ implies $\vec{x}\le 0$.
Conversely, a Z-matrix satisfying $A\vec{x}\le 0\Rightarrow\vec{x}\le 0$ is an
M-matrix.
\end{thm}
\begin{proof}
A is an M-matrix, thus $A^{-1}\ge 0$ by definition. Let $\vec{y} = A\vec{x}$.
Then $\vec{x} = A^{-1}\vec{y}$. The component-wise inequalities $A^{-1}\ge 0$
and $\vec{y}\le 0$ imply $\vec{x}\le 0$. The reverse statement is proved in
\cite[p.~29]{QuarteroniSaccoSaleri2000}.
\end{proof}
\begin{thm}
If $A$ is an M-matrix, and $D$ its diagonal part, then $\rho(I-D^{-1}A)<1$,
thus the Jacobi and the Gau{\ss}-Seidel iteration converge.
\end{thm}
\begin{proof}
The convergence of the Jacobi iteration is given in \cite{Hackbusch1994}.
The Gau{\ss}-Seidel convergence follows from the
Stein-Rosenberg-Theorem \cite{Varga2000}.
\end{proof}
The performance of multigrid methods for meshfree finite difference matrices
has been investigated in \cite{SeiboldMeshfree2007,SeiboldECMI2008}.
Further benefits of an M-matrix structure with respect to linear solvers can
be found in \cite{Varga2000}.

%---------------------------------------------------------------------------------------------
\subsection{A Sufficient Condition for an M-Matrix Structure}
\label{subsec:mmatrix_suff_cond}
%---------------------------------------------------------------------------------------------
Conditions that imply the M-matrix property are required, since the inverse matrix is
typically not directly available.
Let the unknowns be labeled by an index set $I$. We consider square matrices
$A\in\mathbb{R}^{I\times I}$.
\begin{defn}
The \emph{graph} $G(A)$ of a matrix $A$ is defined by
$G(A) = \{(i,j)\in I\times I : a_{ij}\neq 0\}$.
The index $i\in I$ is called \emph{connected} to $j\in I$, if a chain
$i=i_0,i_1,\dots,i_{k-1},i_k=j \ \in I$ exists, such that
$(i_{\nu-1},i_\nu)\in G(A) \ \forall \nu=1,\dots,k$.
\end{defn}
\begin{rem}
For a finite difference matrix, each index $i\in I$ corresponds to a point $\vec{x}_i$.
The index $i\in I$ being connected to $j\in I$ means that the point $\vec{x}_i$
connects (indirectly) to the point $\vec{x}_j$ via stencil entries.
\end{rem}
\begin{defn}
A finite difference matrix is called \emph{essentially irreducible} if every point
is connected to a Dirichlet boundary point.
\end{defn}
\begin{rem}
A finite difference matrix that is not essentially irreducible, is singular,
since the points that are not connected to a Dirichlet point form a singular submatrix.
\end{rem}
\begin{defn}
A matrix $A\in\mathbb{R}^{I\times I}$ is called \emph{essentially diagonally dominant},
if it is weakly diagonally dominant ($\forall i\in I: |a_{ii}|\ge\sum_{k\neq i}|a_{ik}|$),
and every point $i\in I$ is connected to a point $j\in I$ which satisfies the strict
diagonal dominance relation $|a_{jj}|>\sum_{k\neq j}|a_{jk}|$.
\end{defn}
\begin{thm}
\label{seibold:thm:ess_irred_is_ess_diagdom}
An L-matrix arising as a finite difference discretization of
\eqref{seibold:eq:poisson_equation} is essentially diagonally dominant, if it is
essentially irreducible.
\end{thm}
\begin{proof}
For an L-matrix the constant relation in \eqref{seibold:eq:constraints_laplace}
implies the weak diagonal dominance relation for every interior and Neumann point.
Each row corresponding to a Dirichlet point satisfies the strict diagonal dominance
relation.
\end{proof}
\begin{thm}
\label{seibold:thm:Lmatrix_is_Mmatrix}
An essentially diagonally dominant L-matrix is an M-matrix.
\end{thm}
\begin{proof}
The proof is given in \cite[p.~153]{Hackbusch1994}.
\end{proof}
If problem \eqref{seibold:eq:poisson_equation} can be discretized by positive stencils
and every point is connected to a Dirichlet point, then the resulting matrix
is an M-matrix.

%=============================================================================================
\section{Least Squares Approaches}
\label{seibold:sec:least_squares_approaches}
%=============================================================================================
Classical approaches for meshfree derivative approximation are moving least squares
methods, based on scattered data interpolation \cite{LancasterSalkauskas1981},
and local approximation methods, based on generalized finite difference
methods \cite{LiszkaOrkisz1980}. Their application to meshfree settings has
been analyzed in \cite{DuarteLiszkaTworzyako1996,Levin1998}. Differences between
moving and local approaches have been investigated in \cite{SeiboldDiss2006}.

Around a central point $\vec{x}_0$, points inside a radius $r$ are considered.
A distance weight function $w(\delta)$ is defined, which is small for $\delta>r$.
We consider
interpolating\footnote{If $\lim_{\delta\to 0} w(\delta)$ exists, the approach is
called \emph{approximating}, otherwise \emph{interpolating}.}
approaches with $w(\delta)=\delta^{-\alpha}$. Each neighboring point $\vec{x}_i$ is
assigned a weight $w_i = w\prn{\|\vec{x}_i-\vec{x}_0\|_2}$. A unique stencil is
defined via a quadratic minimization problem
\begin{equation}
\min\sum_{i=1}^n\frac{s_i^2}{w_i}, \ \mathrm{s.t.} \
V\cdot\vec{s} = \vec{b}\;.
\label{seibold:eq:quadratic_minimization}
\end{equation}
Using $W=\mathrm{diag}(w_i)$, its solution is
\begin{equation}
\vec{s} = WV^T(VWV^T)^{-1}\cdot\vec{b}\;.
\label{seibold:eq:stencil_lsq}
\end{equation}
After the weights $w_i$ are evaluated, the $k\times k$ matrix $VWV^T$ has to be set up,
the linear system $(VWV^T)\cdot\vec{v}=\vec{b}$ to be solved, and the product
$\vec{s} = WV^T\cdot\vec{v}$ to be computed. This requires $k(k+1)m+\frac{k^3}{3}$
floating point operations \cite[p.~150]{SeiboldDiss2006}.
Least squares approaches do not yield minimal stencils, unless exactly $k$ neighbors
are considered. In general, they also do not yield positive stencils.
\begin{ex}
\label{seibold:ex:qm_nonpos_stencil}
Consider $\vec{x}_0=(0,0)$ and 6 neighbors on the unit
circle $\vec{x}_i=(\cos(\frac{\pi}{2}\varphi_i),\sin(\frac{\pi}{2}\varphi_i))$,
where $(\varphi_1,\dots,\varphi_6)=(0,1,2,3,0.1,0.2)$
(see Fig.~\ref{seibold:fig_lsq_nonpos_stencil_qm}).
Since all neighbors have the same distance from $\vec{x}_0$, the distance weight
function does not play a role. Formula \eqref{seibold:eq:stencil_lsq} yields the
non-positive least squares stencil $\vec{s}=(0.846,1.005,0.998,1.003,0.312,-0.164)$.
However, the configuration admits a positive stencil, namely $\vec{s}=(1,1,1,1,0,0)$.
\end{ex}

%=============================================================================================
\section{Linear Minimization Approach}
\label{seibold:sec:lp_approach}
%=============================================================================================
Least squares approaches do not guarantee positive stencils. As motivated in
Sect.~\ref{seibold:subsec:min_pos_stencils}, we wish to allow positive stencils only.
Hence, we enforce positivity, i.e.~we search for solutions in the polyhedron
\begin{equation}
P = \{\vec{s}\in\mathbb{R}^m : V\cdot\vec{s} = \vec{b} , \
\vec{s}\ge 0\}\;.
\label{seibold:eq:polyhedron}
\end{equation}
This is the feasibility problem of linear optimization \cite{Vanderbei2001}.
In Sect.~\ref{seibold:sec:lp_conditions_positive_stencil} we show under which
conditions solutions exist.
If $P$ is nonvoid and not degenerate, there are infinitely many feasible stencils.
To single out a unique stencil we formulate a \emph{linear} minimization problem
\begin{equation}
\min \sum_{i=1}^m \frac{s_i}{w_i}, \ \mathrm{s.t.} \
V\cdot\vec{s}=\vec{b}, \ \vec{s}\ge 0\;,
\label{seibold:eq:linear_minimization}
\end{equation}
where the weights $w_i=w(\|\vec{x}_i-\vec{x}_0\|)$ are defined by an appropriately
decaying (see Thm.~\ref{seibold:thm:mps_distance_weight_function}) non-negative
distance weight function $w$.
Problem \eqref{seibold:eq:linear_minimization} is a linear program (LP) in standard
form. It is bounded, since we have imposed sign constraints and the weights $w_i$
are all non-negative.
\begin{thm}
If the polyhedron \eqref{seibold:eq:polyhedron} is nonvoid, then the linear
minimization approach \eqref{seibold:eq:linear_minimization} yields minimal positive
stencils.
\end{thm}
\begin{proof}
The sign constraints in \eqref{seibold:eq:linear_minimization} ensure that the selected
stencil is \emph{positive}. The existence of a \emph{minimal} solution is ensured by
the \emph{fundamental theorem of linear programming} \cite{Vanderbei2001}.
If the LP \eqref{seibold:eq:linear_minimization} has a solution, then it also has a
basic solution, in which at most $k$ of the $m$ stencil entries $s_i$ are different
from zero.
\end{proof}
In contrast to least squares methods, the LP approach yields nonzero values only for
a few selected points, and a continuous dependence on the point positions is not
possible (see Rem.~\ref{seibold:rem:no_continuous_dependence}).
\begin{rem}
One could ask why not remain with a least squares problem, and additionally impose
sign constraints. While this would be a valid approach (the solution is obtained by
Karush-Kuhn-Tucker methods \cite{Vanderbei2001}), it would have the worst of both
worlds. The solution would not depend continuously on the point cloud geometry
whenever the sign constraints are active, and the resulting stencil would not be
minimal. When sign constraints are imposed, linear minimization is preferable.
\end{rem}

\begin{figure}
 \centering
 \begin{minipage}[t]{.30\textwidth}
  \centering
  \includegraphics[width=0.99\textwidth]{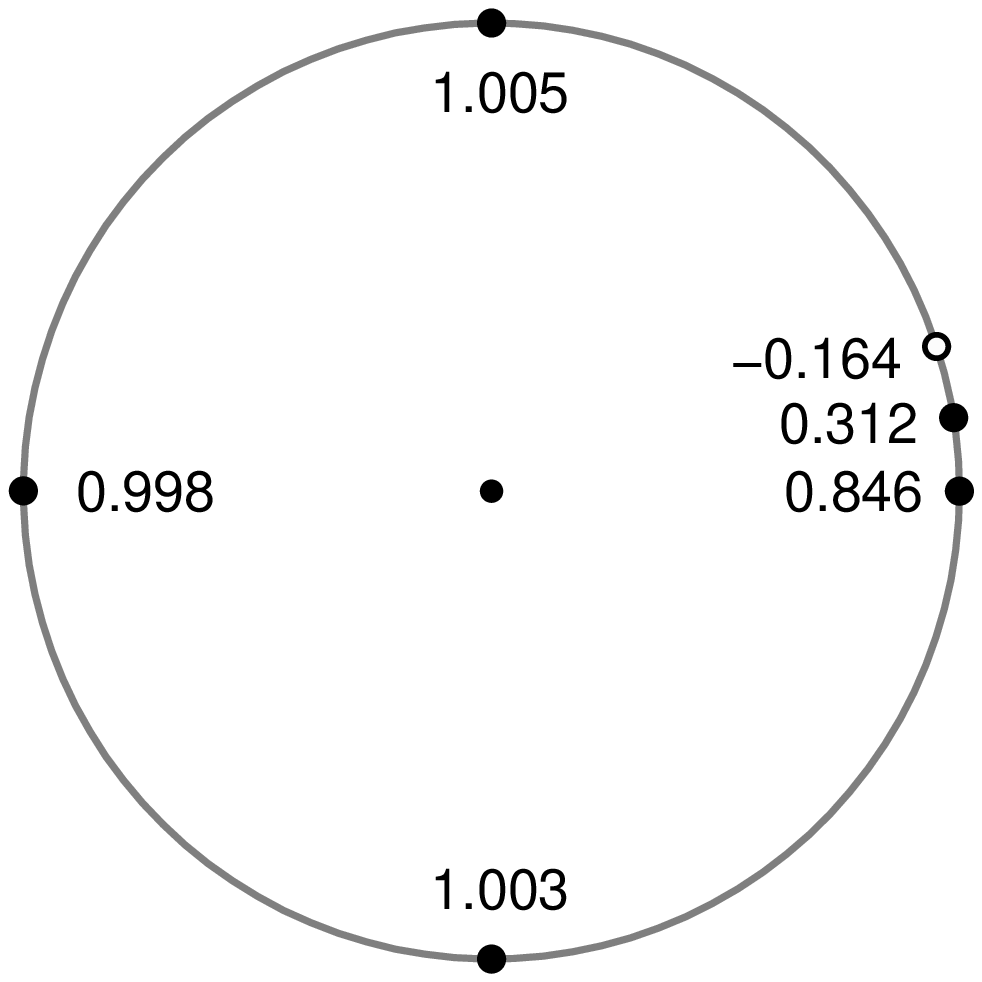}
  \caption{Non-positive LSQ stencil}
  \label{seibold:fig_lsq_nonpos_stencil_qm}
 \end{minipage}
 \hfill
 \begin{minipage}[t]{.67\textwidth}
  \centering
  \begin{minipage}[t]{.32\textwidth}
   \includegraphics[width=0.99\textwidth]{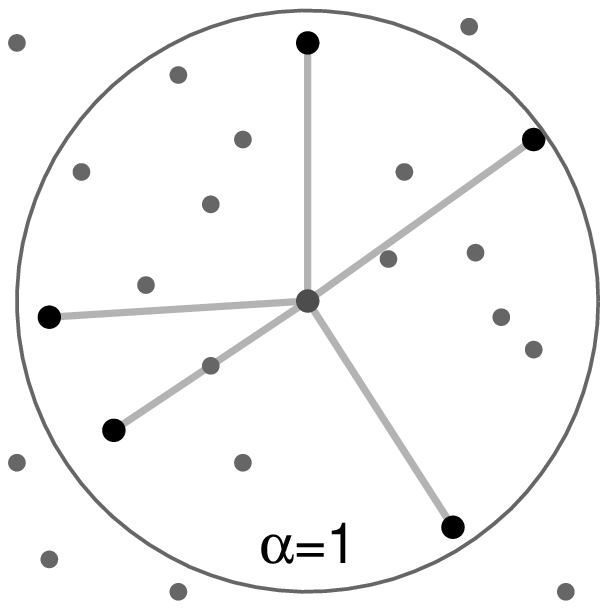}
  \end{minipage}
  \hfill
  \begin{minipage}[t]{.32\textwidth}
   \includegraphics[width=0.99\textwidth]{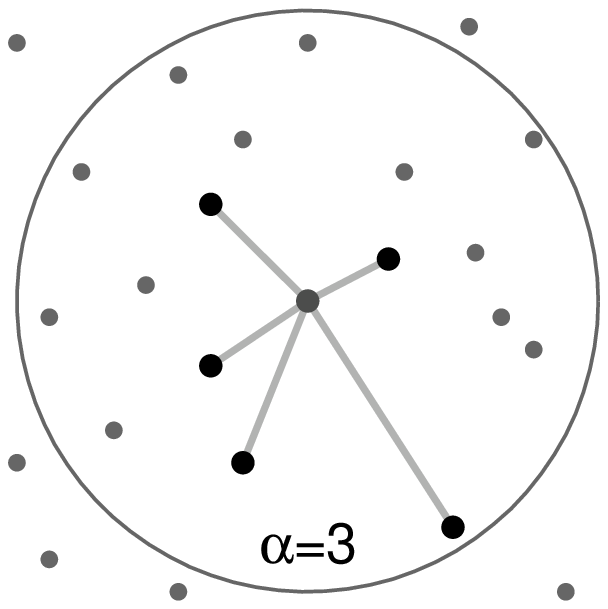}
  \end{minipage}
  \hfill
  \begin{minipage}[t]{.32\textwidth}
   \includegraphics[width=0.99\textwidth]{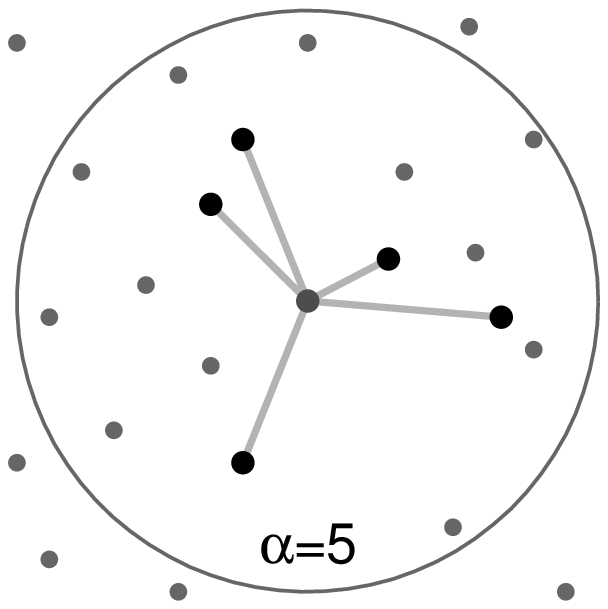}
  \end{minipage}
  \caption{Minimal positive stencil for various values of $\alpha$}
  \label{seibold:fig_stencil_weight_min_alpha}
 \end{minipage}
\end{figure}

%---------------------------------------------------------------------------------------------
\subsection{Solving the Linear Programs}
\label{subsec:solving_lp}
%---------------------------------------------------------------------------------------------
For every interior point consider a set of candidate points ($m>k$). A basic solution
of \eqref{seibold:eq:linear_minimization} is computed. Only the nonzero stencil values
enter the Poisson matrix.
We refer to this approach as \emph{minimal positive stencil} (MPS) method.
The LPs \eqref{seibold:eq:linear_minimization} are small, but they have to be solved
for every interior point. To our knowledge, there are no general results about
efficient methods for such small LPs, especially considering the special structure
of the Vandermonde matrix. A numerical comparison of various methods has been presented
in \cite[p.~148]{SeiboldDiss2006}. Simplex methods perform best for the arising LPs.
A basis change corresponds to one stencil point replacing another. The theoretical worst
case performance of simplex methods is not observed. Typical runs find the solution in
about $1.5 k$ steps, resulting in a complexity of $O(k^2m)$, which equals the
effort of least squares approaches (see Sect.~\ref{seibold:sec:least_squares_approaches}).

%---------------------------------------------------------------------------------------------
\subsection{Geometric Interpretation of Minimal Positive Stencils}
\label{subsec:mps_laplace_geometric_interpretation}
%---------------------------------------------------------------------------------------------
The MPS method forms a compromise between selecting neighbors close by and distributed
nicely (see Def.~\ref{seibold:def:points_distributed_nicely}) around the central point.
How much preference is given to which objective depends on the locality parameter
$\alpha$ in the distance weight function $w(\delta) = |\delta|^{-\alpha}$.
\begin{thm}
\label{seibold:thm:mps_distance_weight_function}
The MPS method \eqref{seibold:eq:linear_minimization} only leads to reasonable results,
if the distance weight function decays faster than $|\delta|^{-2}$.
\end{thm}
\begin{proof}
Summing over the diagonal in the quadratic constraints in
\eqref{seibold:eq:constraints_laplace} yields the relation
$\sum_{i=1}^m\|\vec{\bar x}_i\|_2^2s_i = 2d$.
If $w(\delta)$ decays faster than $|\delta|^{-2}$, points close to the central point
are given preference.
If $w(\delta)=|\delta|^{-2}$, the LP \eqref{seibold:eq:linear_minimization} is
degenerate.
If $w(\delta)$ decays slower than $|\delta|^{-2}$, the approach selects
points far away from the central point, possibly resulting in ``checkerboard''
instabilities.
\end{proof}
The dependence of the MPS stencil on $\alpha$ is shown
in Fig.~\ref{seibold:fig_stencil_weight_min_alpha}.\footnote{The figures are 2d
for simplicity of presentation. The MPS method applies directly
to, and shows its strength, in 3d.}
From the candidate points in the circle, five neighbors are selected.
While for $\alpha=1$ far away points are selected, $\alpha\in\{3,5\}$ yields
nearby points. For $\alpha=3$ smaller angles are more important, for $\alpha=5$
smaller distances. Note that the MPS method never
selects neighbors which are not distributed around the central point (as defined in
Sect.~\ref{seibold:sec:lp_conditions_positive_stencil}), even if those are the $k$
closest points.
\begin{rem}
For regular grids, the MPS method selects standard finite difference stencils.
For instance, for a regular Cartesian grid, the standard 5-point (2d), respectively
7-point (3d) stencils are obtained. In these cases, the basic solution is degenerate,
i.e.~some of the basis variables are zero.
\end{rem}

%---------------------------------------------------------------------------------------------
\subsection{Neighborhood Criteria}
\label{seibold:subsec:neighborhood}
%---------------------------------------------------------------------------------------------
The circular neighborhood criterion yields a large number of neighbors (unless $r$ is
very small). Also, it does not guarantee positive stencils, as the example in
Fig.~\ref{seibold:fig_neighborhood_circular} shows. The selected neighbors are
marked bold. Neighbors with non-positive stencil values are indicated by a white center.
The presented methods can also be based on other neighborhood criteria.

Defining a neighborhood via the a Delaunay triangulation (tetrahedrization in 3d)
\cite{Shewchuk1999}, yields significantly fewer neighbors. However,
the construction is a meshing procedure, hence it is often undesirable in a meshfree
context.\footnote{Hybrid approaches exist that use a Delaunay mesh combined
with meshfree methods.}
Fig.~\ref{seibold:fig_neighborhood_delaunay} shows that a Delaunay neighborhood
constructed from a Voronoi tessellation may yield non-positive stencils. In addition,
it is possible that not enough neighbors are selected. If, in the example, the two
negative points were removed, only four neighbors would be defined.

The \emph{four quadrant criterion} \cite{DuarteLiszkaTworzyako1996}
(eight sectors in 3d) defines a local coordinate system and selects the two closest
points from each sector. It guarantees the neighbors to be distributed around the
central point. However, it does not guarantee positive stencils, as the example in
Fig.~\ref{seibold:fig_neighborhood_four_quadrant} shows.

\begin{figure}
\centering
\begin{minipage}[t]{.23\textwidth}
\centering
\includegraphics[width=0.99\textwidth]{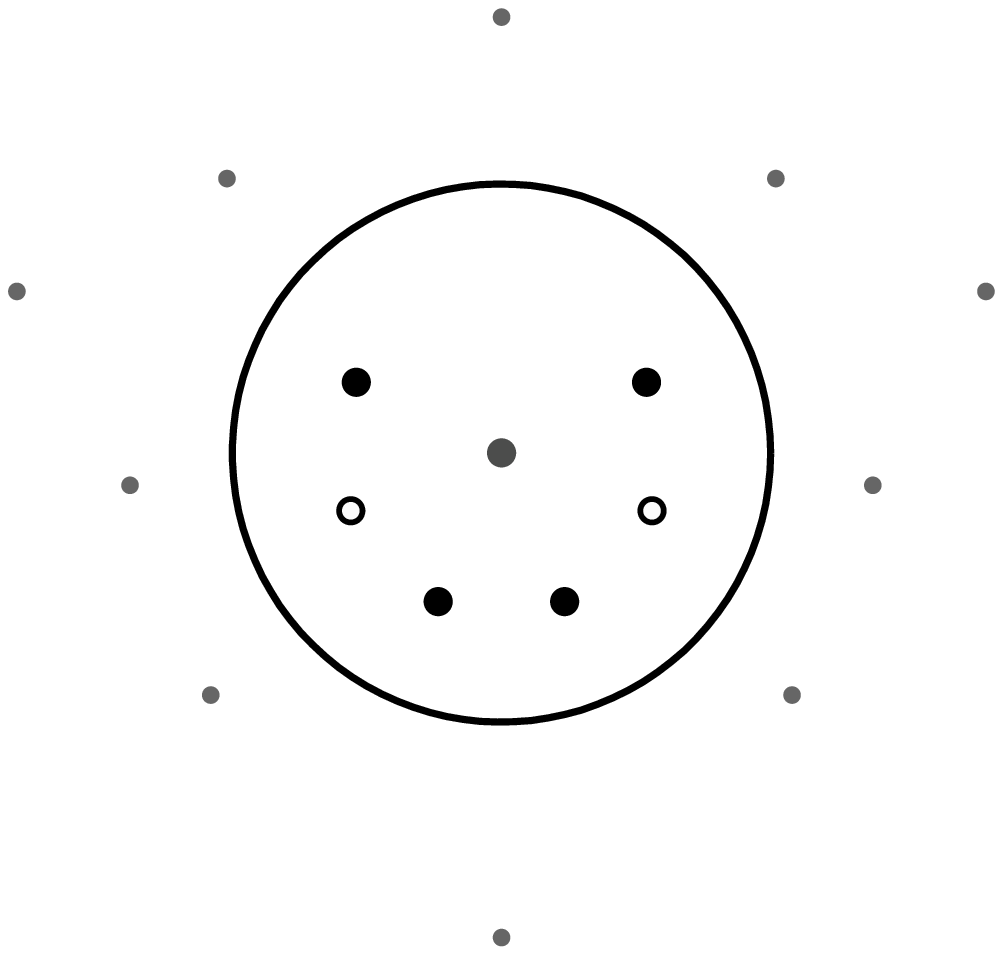}
\caption{Circular neighborhood}
\label{seibold:fig_neighborhood_circular}
\end{minipage}
\hfill
\begin{minipage}[t]{.23\textwidth}
\centering
\includegraphics[width=0.99\textwidth]{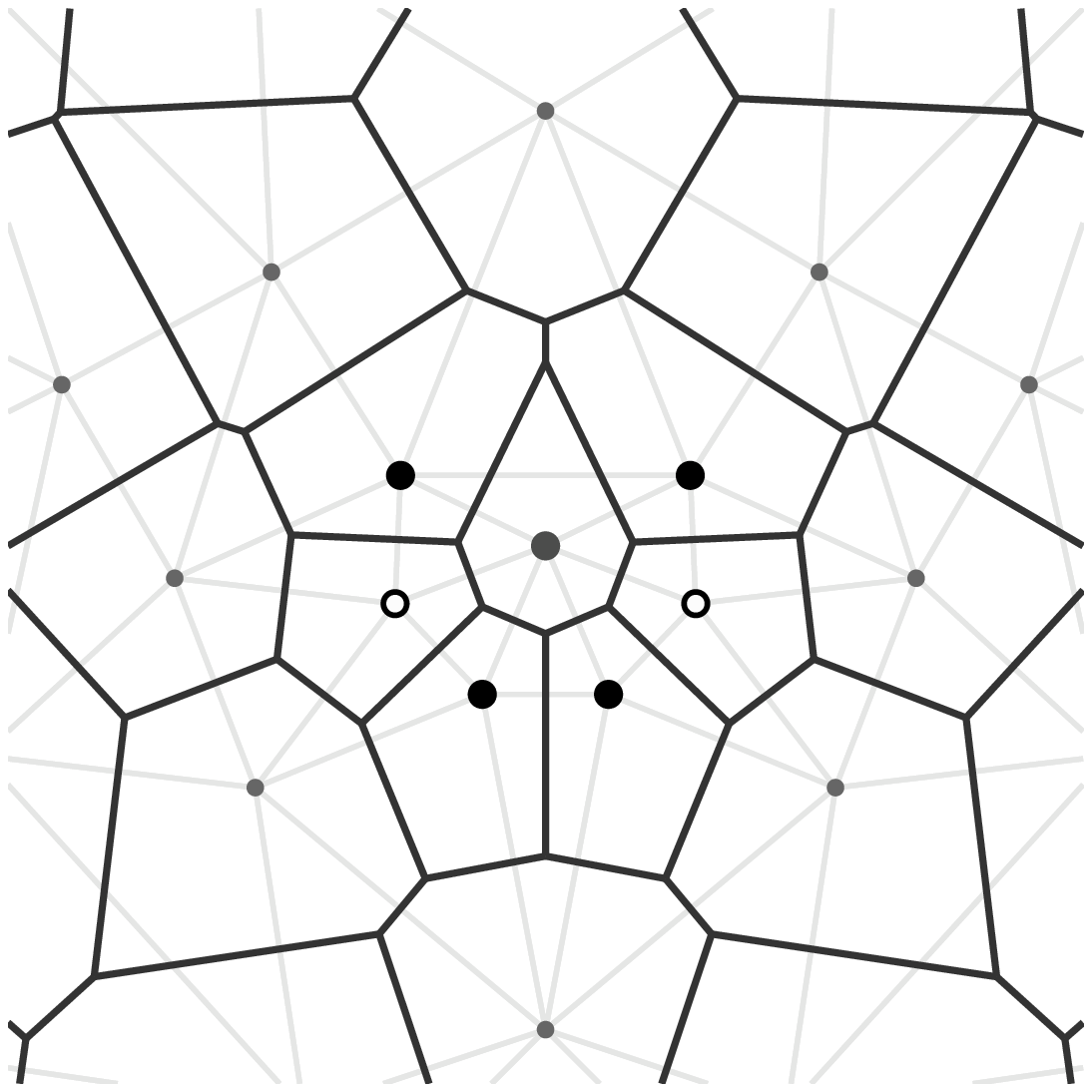}
\caption{Delaunay neighbors}
\label{seibold:fig_neighborhood_delaunay}
\end{minipage}
\hfill
\begin{minipage}[t]{.23\textwidth}
\centering
\includegraphics[width=0.99\textwidth]{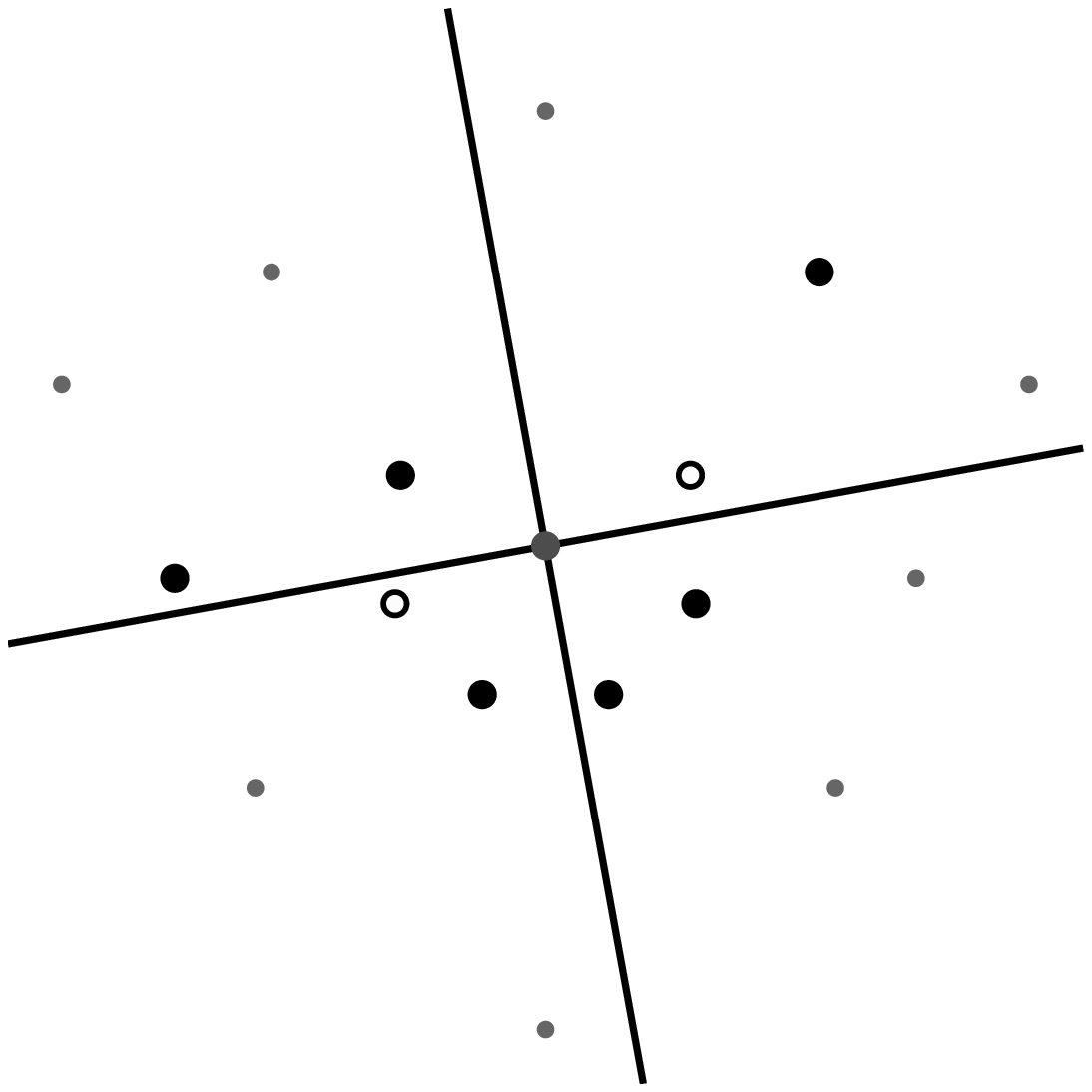}
\caption{Four quadrants}
\label{seibold:fig_neighborhood_four_quadrant}
\end{minipage}
\hfill
\begin{minipage}[t]{.23\textwidth}
\centering
\includegraphics[width=0.99\textwidth]{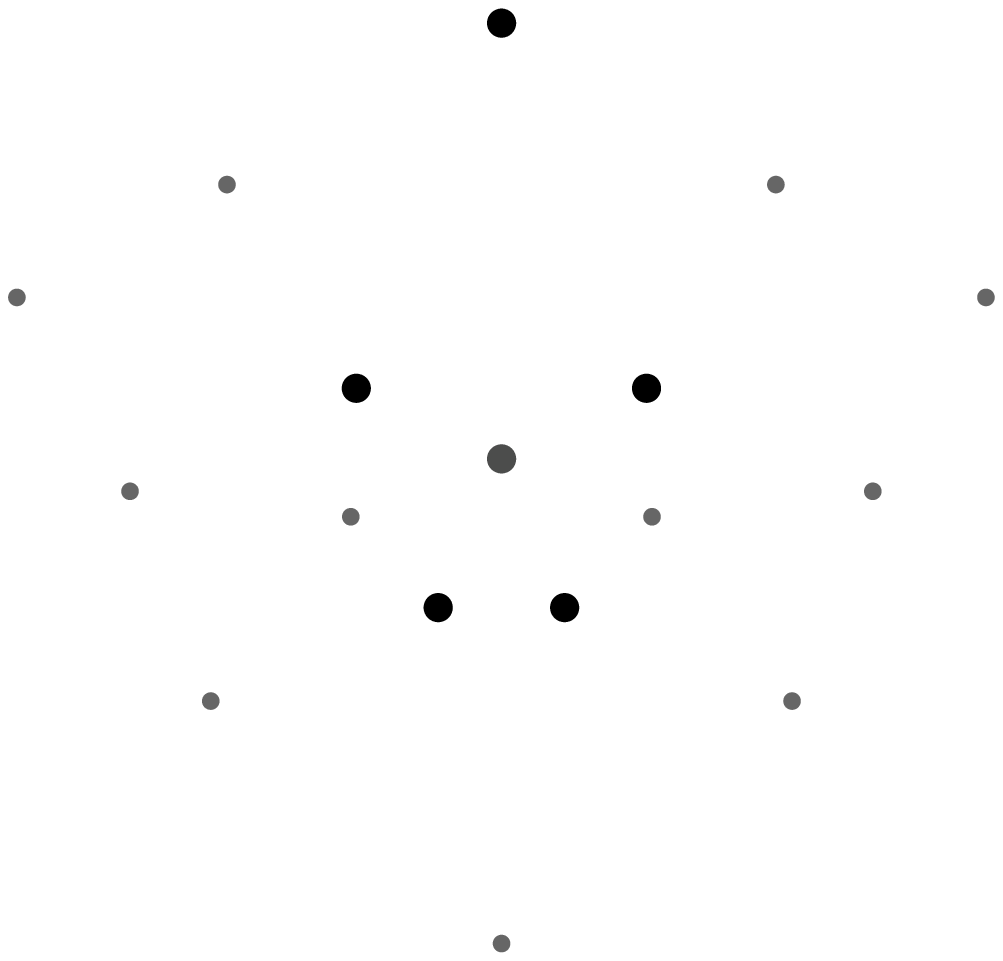}
\caption{MPS neighborhood}
\label{seibold:fig_neighborhood_pos_stencil}
\end{minipage}
\end{figure}

The stencil selected by the MPS method is shown in
Fig.~\ref{seibold:fig_neighborhood_pos_stencil}. It is minimal and positive,
here achieving this property by selecting one point further away.
The MPS method can be interpreted as a neighborhood criterion that is optimal
(i.e.~minimal and positive) with respect to the Laplace operator.
A deeper discussion of various neighborhood criteria in presented
in \cite{SeiboldDiss2006}.

%=============================================================================================
\section{Conditions for the Existence of Positive Stencils}
\label{seibold:sec:lp_conditions_positive_stencil}
%=============================================================================================
We investigate when the polyhedron \eqref{seibold:eq:polyhedron} is nonvoid,
i.e.~under which conditions positive stencils exist. We place the point of
approximation in the origin $\vec{x}_0=\vec{0}$. A set of neighbors
$\{\vec{x}_1,\dots,\vec{x}_m\}\subset\mathbb{R}^d$ is given, where $m\ge k$.
In order to establish a connection between the LP space $\mathbb{R}^m$ and the actual
geometry space $\mathbb{R}^d$ we consider the dual problem, which is defined in
$\mathbb{R}^k$.
\begin{thm}[Farkas' Lemma]
For a real matrix $A$ and a real vector $\vec{b}$, exactly one of the following two
systems has a solution:
\begin{itemize}
\item $A\cdot\vec{x}=\vec{b}$ for some $\vec{x}\ge 0$, or
\item $A^T\cdot\vec{w}\ge 0$ for some $\vec{w}$ satisfying $\vec{b}^T\cdot\vec{w}<0$.
\end{itemize}
\end{thm}
\begin{proof}
The proof is given in \cite{Vanderbei2001}.
\end{proof}
Applying Farkas' lemma to our problem yields that system
$V\cdot\vec{s}=\vec{b}$ has no solution $\vec{s}\ge 0$, if and only if system
$V^T\cdot\vec{w}\ge 0$ has a solution satisfying $\vec{b}^T\cdot\vec{w}<0$.
The $i^{th}$ component of $V^T\cdot\vec{w}$ can be written as
\begin{equation*}
\prn{V^T\cdot\vec{w}}_i = \vec{a}^T\cdot\vec{x}_i+\vec{x}_i^T\cdot A\cdot\vec{x}_i\;,
\end{equation*}
where $\vec{a}=\prn{w_1,\dots,w_d}^T$ and $A$ is the symmetric matrix
\begin{equation*}
A = \prn{\begin{array}{cc}
w_4 & w_3 \\
w_3 & w_5 \\
\end{array}}
\ \mathrm{(2d)}
\quad\mathrm{resp.}\quad
A = \prn{\begin{array}{ccc}
w_7 & w_4 & w_5 \\
w_4 & w_8 & w_6 \\
w_5 & w_6 & w_9 \\
\end{array}}
\ \mathrm{(3d)}\;.
\end{equation*}
Given $\vec{w}$ (respectively $\vec{a}$ and $A$), we consider the quadratic form
$f(\vec{x}) = \vec{a}^T\cdot\vec{x}+\vec{x}^T\cdot A\cdot\vec{x}$.
Since $A$ is symmetric, an orthogonal matrix $S\in O(d)$ exists, such that
$S^TAS = D$, where $D=\mathrm{diag}\prn{\lambda_1,\dots,\lambda_d}$. In the new
coordinates, with $\vec{d} = S^T\vec{a}$, we define
\begin{equation}
g(\vec{x}) = f(S\vec{x}) = \vec{d}^T\cdot\vec{x}+\vec{x}^T\cdot D\cdot\vec{x}\;.
\label{seibold:eq:pos_stencil_function_d}
\end{equation}
If all eigenvalues $\lambda_i\neq 0$, then $D$ is regular.
With $\vec{c} = -\frac{1}{2}D^{-1}\vec{d}$ we can write
\begin{equation*}
g(\vec{x}) = (\vec{x}-\vec{c})^T\cdot D
\cdot (\vec{x}-\vec{c})-\vec{c}^T\cdot D\cdot\vec{c}\;,
\end{equation*}
If one or two $\lambda_i=0$, we are in a degenerate case, and stick to the
representation with $\vec{d}$ as parameter. Choosing $\vec{w}\in\mathbb{R}^k$
arbitrarily is equivalent to choosing $S\in O(d)$,
$\vec{\lambda}=\prn{\lambda_1,\dots,\lambda_d}\in\mathbb{R}^d$ and
$\vec{c}\in\mathbb{R}^d$ (respectively $\vec{d}\in\mathbb{R}^d$) arbitrarily.
For any $\vec{\lambda},\vec{c}\in\mathbb{R}^d$ define the domain
\begin{equation*}
H_{\vec{\lambda},\vec{c}} = \{\vec{x}\in\mathbb{R}^d:g(\vec{x})\ge 0\}\;.
\end{equation*}
For a set of points $X=\{\vec{x}_1,\dots,\vec{x}_m\}$ define
$SX=\{S\vec{x}_1,\dots,S\vec{x}_m\}$. Farkas' lemma translates to
\begin{cor}
\label{seibold:cor:pos_stencil_geometric_criterion}
System
$V\cdot\vec{s}=\vec{b}$ has no solution $\vec{s}\ge 0$, if and only if $S\in O(d)$,
$\vec{c},\vec{\lambda}\in\mathbb{R}^d$ with $\sum_{i=1}^d\lambda_i<0$ exist, such
that $SX\subset H_{\vec{\lambda},\vec{c}}$.
\end{cor}
In other words, no positive Laplace stencil exists, iff the set of points $X$ can
be transformed (via $S\in O(d)$), such that it is contained in the set
$H_{\vec{\lambda},\vec{c}}$ for some $\vec{c},\vec{\lambda}\in\mathbb{R}^d$ with
$\sum_{i=1}^d\lambda_i<0$.

\begin{figure}
\centering
\begin{minipage}[t]{.28\textwidth}
\centering
\includegraphics[width=0.81\textwidth]{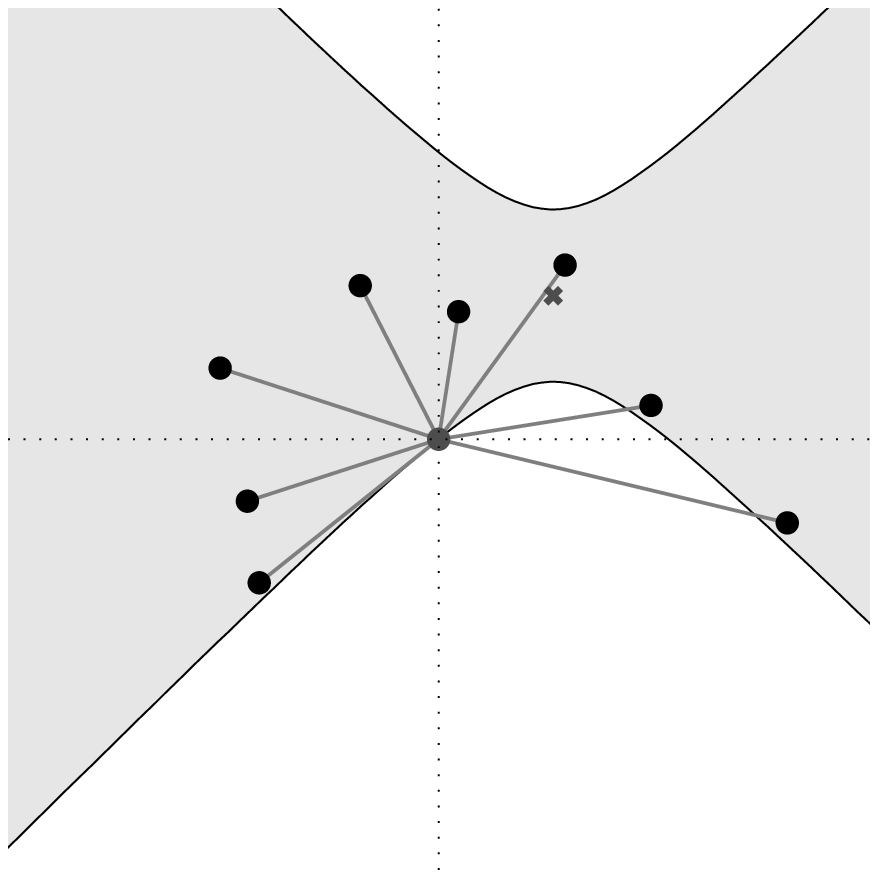}
\caption{No positive stencil exists}
\label{seibold:fig_proof_pos_stencil_2d_points_neg}
\end{minipage}
\hfill
\begin{minipage}[t]{.28\textwidth}
\centering
\includegraphics[width=0.81\textwidth]{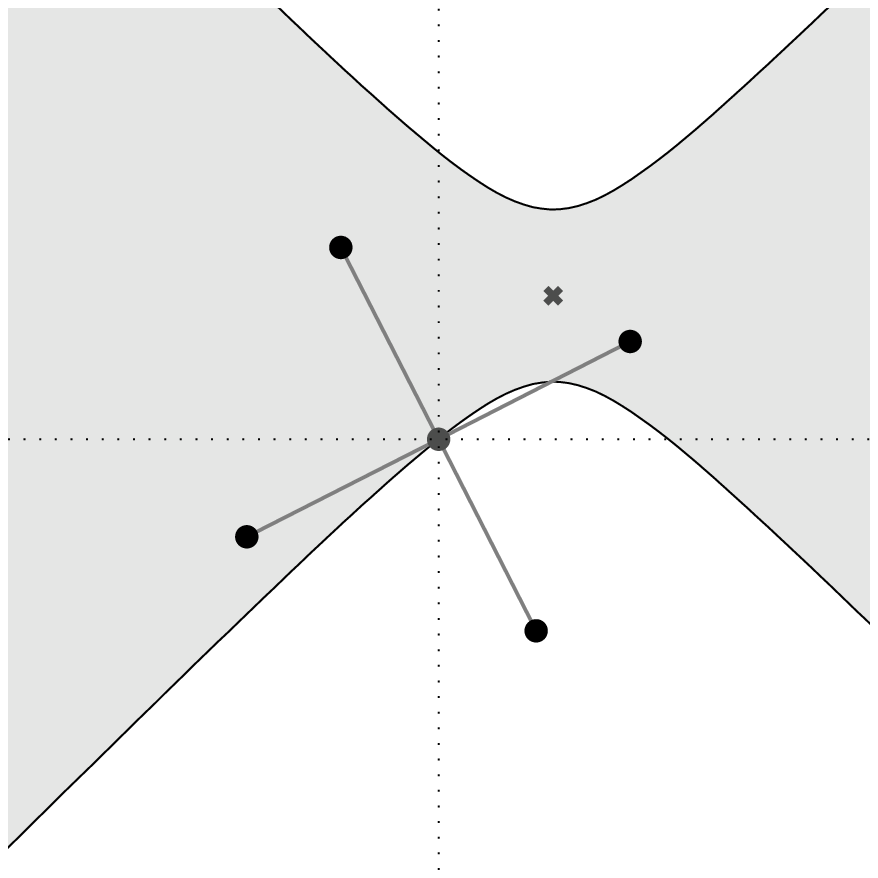}
\caption{Positive stencil}
\label{seibold:fig_proof_pos_stencil_2d_points_pos}
\end{minipage}
\hfill
\begin{minipage}[t]{.28\textwidth}
\centering
\includegraphics[width=0.81\textwidth]{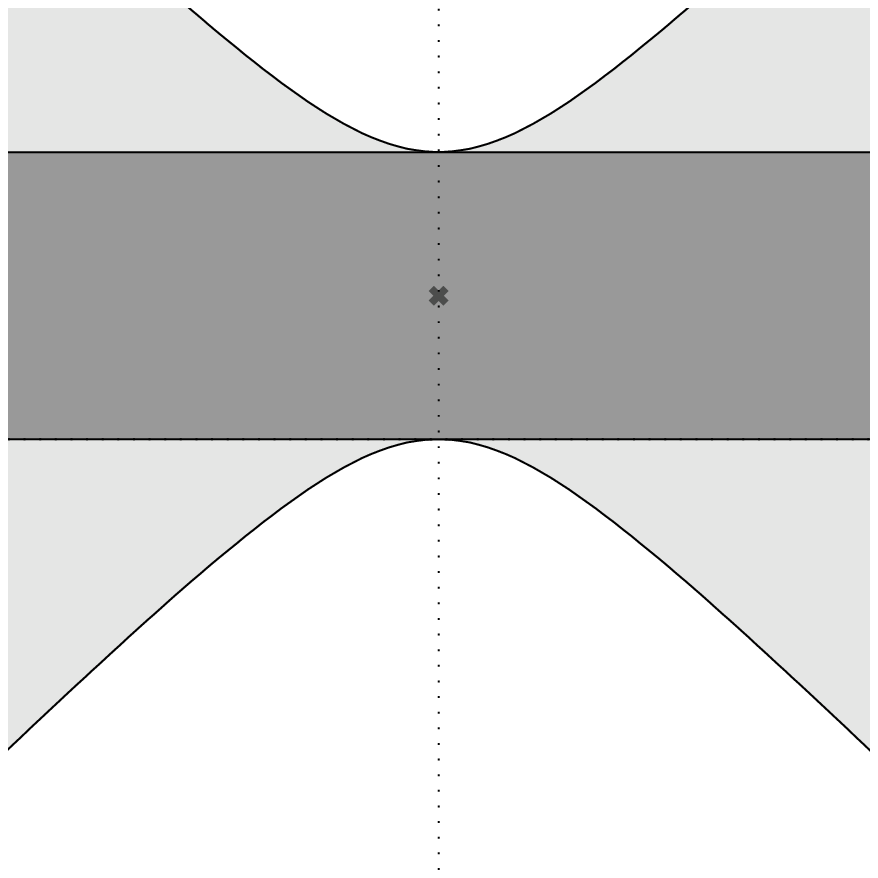}
\caption{Necessary criterion}
\label{seibold:fig_proof_pos_stencil_2d_necc}
\end{minipage}
\end{figure}

\begin{ex}
The setup in Fig.~\ref{seibold:fig_proof_pos_stencil_2d_points_neg} shows a set of
points that is completely contained in a domain $H_{\vec{\lambda},\vec{c}}$. Due to
Cor.~\ref{seibold:cor:pos_stencil_geometric_criterion}, no positive stencil exists.
\end{ex}
\begin{ex}
The setup in Fig.~\ref{seibold:fig_proof_pos_stencil_2d_points_pos} shows a point
setup which has a positive stencil solution. It is impossible to find a domain
$H_{\vec{\lambda},\vec{c}}$ and to rotate the set of points, such that all points
are contained in the domain.
\end{ex}
We have derived a geometric condition, which is equivalent to the existence of
positive stencils. However, due to the nonlinearity in $g$, it is difficult to
translate into geometric means directly. Instead, we derive a necessary (but not
sufficient) as well as a sufficient (but not necessary) criterion on the point
geometry for the existence of a positive Laplace stencil. To our knowledge the
latter has not been given yet.

%---------------------------------------------------------------------------------------------
\subsection{A Necessary Criterion for Positive Stencils}
%---------------------------------------------------------------------------------------------
If $V\cdot\vec{s}=\vec{b}$ has a solution $\vec{s}\ge 0$, then for any
$S\in O(d)$, $\vec{c},\vec{\lambda}\in\mathbb{R}^d$ with $\sum_{i=1}^d\lambda_i<0$,
there is a point $\vec{x}_i$ with $S\vec{x}_i\notin H_{\vec{\lambda},\vec{c}}$.
For the particular choice $\lambda_1=-1$, $\lambda_i=0 \ \forall i>1$ and
$c_1\gg\max_i\|\vec{x}_i\|$ it follows that for any $S\in O(d)$ at least one
point must satisfy $x_1<0$. This yields the following
\begin{thm}
\label{seibold:thm:mps_pos_stencil_necessary}
If a set of points $X\subset\mathbb{R}^d$ around the origin admits a positive
Laplace stencil, then they must not lie in one and the same half space (with
respect to an arbitrary hyperplane through the origin).
\end{thm}
This result is well known \cite{DuarteLiszkaTworzyako1996}. Due to the
particular choice of $\vec{\lambda}$, this criterion is very crude, but easy to
formulate in geometric means. More careful estimates of the condition of
Cor.~\ref{seibold:cor:pos_stencil_geometric_criterion} may yield stricter criteria.

\begin{figure}
\centering
\begin{minipage}[t]{.23\textwidth}
\centering
\includegraphics[width=0.99\textwidth]{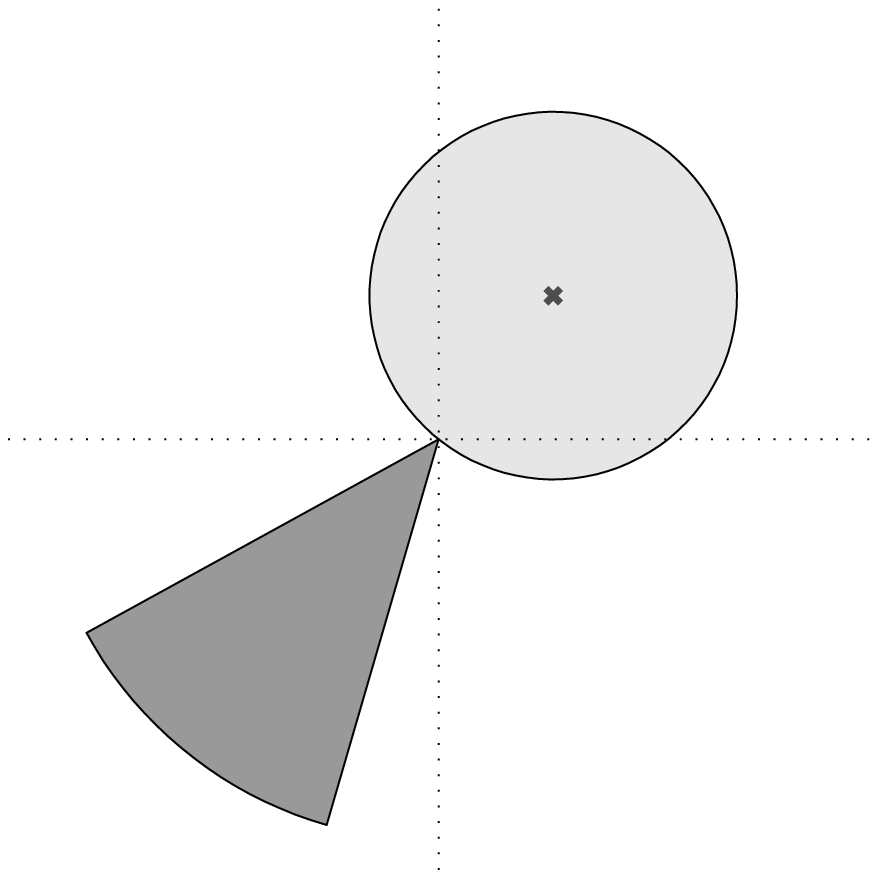}
\caption{Case $(--)$}
\label{fig_proof_pos_stencil_2d_mm}
\end{minipage}
\hfill
\begin{minipage}[t]{.23\textwidth}
\centering
\includegraphics[width=0.99\textwidth]{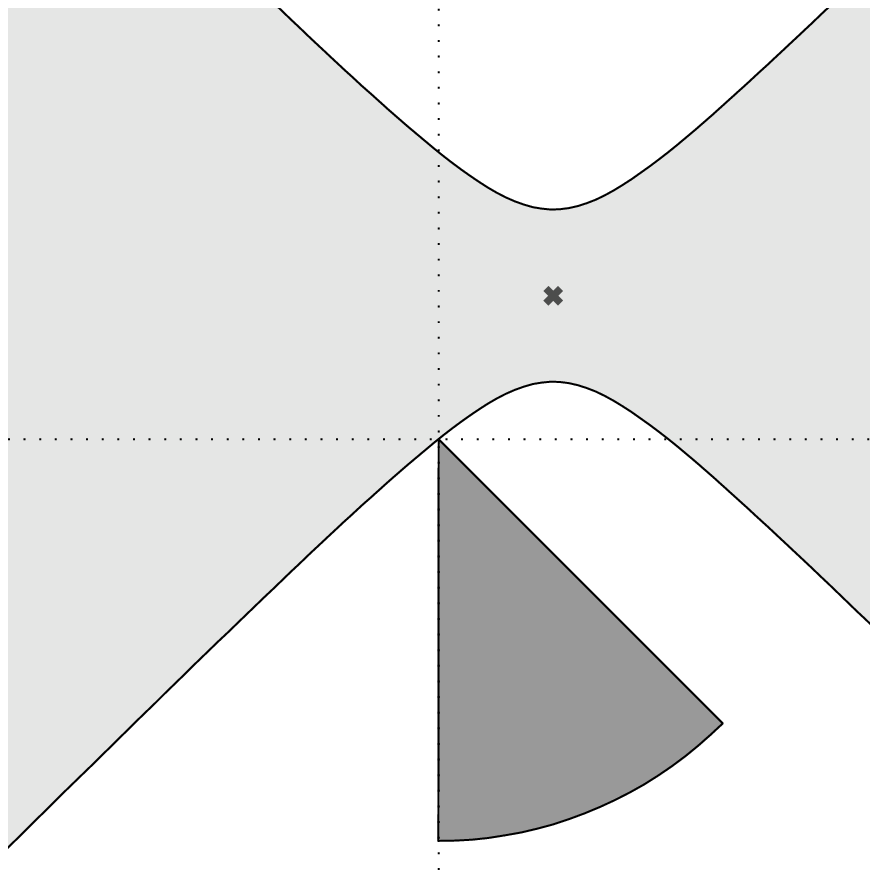}
\caption{Case $(+-)$}
\label{fig_proof_pos_stencil_2d_pm}
\end{minipage}
\hfill
\begin{minipage}[t]{.23\textwidth}
\centering
\includegraphics[width=0.99\textwidth]{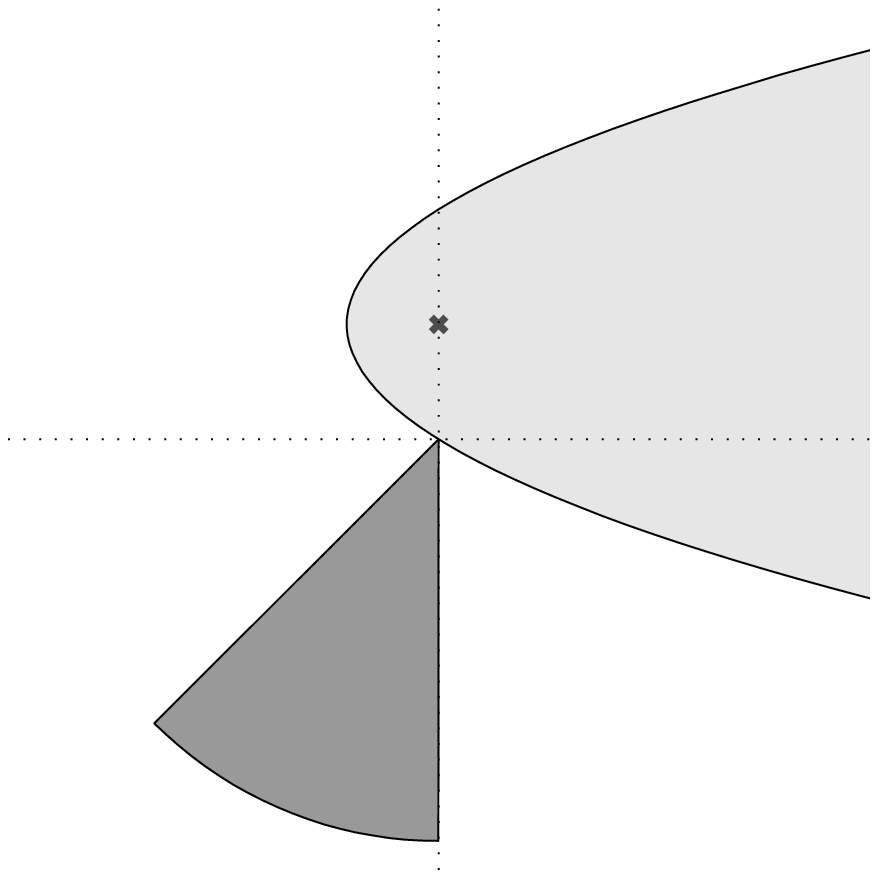}
\caption{Case $(0-)$ type 1}
\label{fig_proof_pos_stencil_2d_zm1}
\end{minipage}
\hfill
\begin{minipage}[t]{.23\textwidth}
\centering
\includegraphics[width=0.99\textwidth]{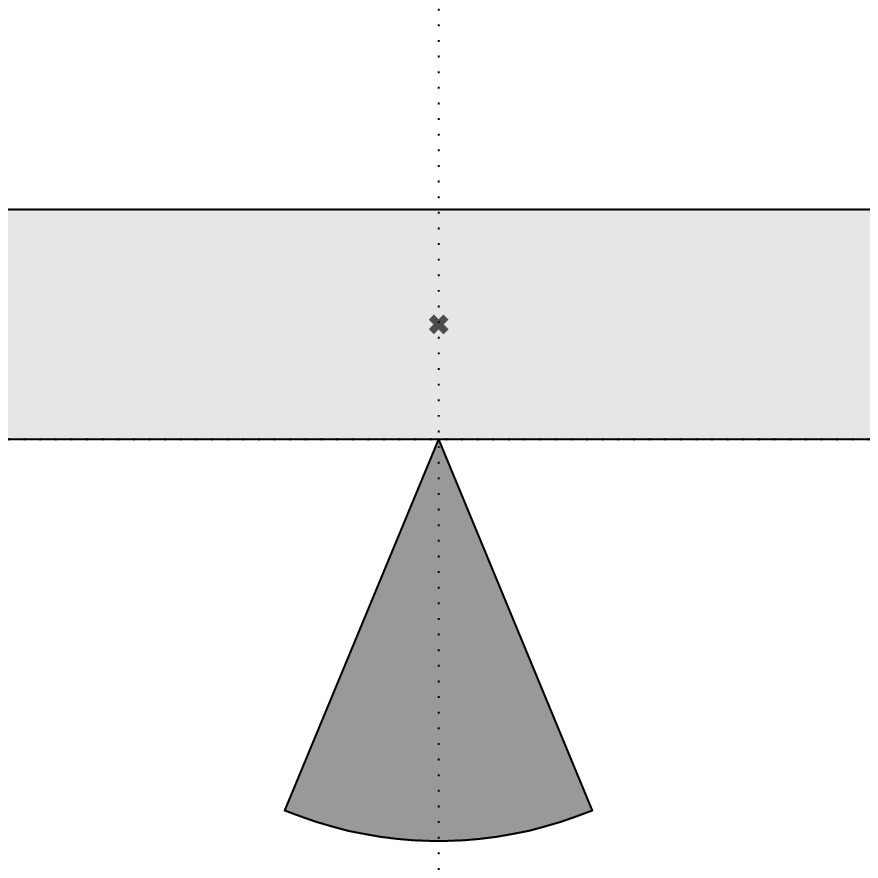}
\caption{Case $(0-)$ type 2}
\label{fig_proof_pos_stencil_2d_zm2}
\end{minipage}
\end{figure}

%---------------------------------------------------------------------------------------------
\subsection{A Sufficient Criterion for Positive Stencils}
%---------------------------------------------------------------------------------------------
For any $\vec{c},\vec{\lambda}\in\mathbb{R}^d$ with $\sum_{i=1}^d\lambda_i<0$
we construct a domain $G_{\vec{\lambda},\vec{c}}\supset H_{\vec{\lambda},\vec{c}}$,
which is $\mathbb{R}^d$ aside from a cone centered at the origin. If for any
$\vec{c},\vec{\lambda}\in\mathbb{R}^d$, $S\in O(d)$ there is at least one point
$S\vec{x}_i\notin G_{\vec{\lambda},\vec{c}}$, then
$S\vec{x}_i\notin H_{\vec{\lambda},\vec{c}}$, thus a positive Laplace stencil exists.
We call this criterion \emph{cone criterion}.
\begin{thm}[Cone criterion in 2d]
\label{seibold:thm:pos_stencil_domain2d}
Let $\vec{c},\vec{\lambda}\in\mathbb{R}^2$ with $\lambda_1+\lambda_2<0$.
There exists always a cone $C_{\vec{v}}$ defined by
$\vec{v}\cdot\vec{x}>\frac{1}{\sqrt{1+\beta^2}}\|\vec{x}\|$,
where $\beta=\sqrt{2}-1$ (a cone with total opening angle 45\ensuremath{^\circ},
the direction vector $\vec{v}$ depends on $\vec{\lambda}$ and $\vec{c}$),
such that $G_{\vec{\lambda},\vec{c}} = \mathbb{R}^d \setminus C_{\vec{v}}$
satisfies $H_{\vec{\lambda},\vec{c}}\subset G_{\vec{\lambda},\vec{c}}$.
\end{thm}
\begin{proof}
We show that $H_{\vec{\lambda},\vec{c}}$ and $C_{\vec{v}}$ do not intersect. Since
the problem is invariant under interchanging coordinates, we can w.l.o.g.~assume
that $\lambda_2<0$. Including the degenerate case, three cases need to be considered:
\begin{itemize}
\item \textbf{Case $(--)$:} $\lambda_1<0$, $\lambda_2<0$ \\
The set $H_{\vec{\lambda},\vec{c}}$ is the interior of an ellipse centered at
$\vec{c}$ with $\vec{0}\in\partial H_{\vec{\lambda},\vec{c}}$.
The vector $\vec{v}=-(\frac{\lambda_2}{\lambda_1}c_1,\frac{\lambda_1}{\lambda_2}c_2)$
is an outer normal vector. The cone $C_{\vec{v}}$ touches the ellipse only at the
origin, as shown in Fig.~\ref{fig_proof_pos_stencil_2d_mm}.
\item \textbf{Case $(+-)$:} $\lambda_1>0$, $\lambda_2<0$ \\
Fig.~\ref{fig_proof_pos_stencil_2d_pm} shows the geometry.
Define $\mu_1 = \frac{|\lambda_1|}{|\lambda_2|}<1$. The domain
$H_{\vec{\lambda},\vec{c}}$ is defined by
$\tilde g(x_1,x_2) = \mu_1(x_1^2-2c_1x_1)-(x_2^2-2c_2x_2) \ge 0$.
Due to symmetry we can assume $c_1,c_2\ge 0$. For all $\vec{x}\in B$, where
$B=\{(x_1,x_2)|x_1>0,x_2<0,|x_1|<|x_2|\}$, the function $\tilde g$ satisfies
\begin{align*}
\tilde g(\vec{x}) &= \mu_1(|x_1|^2-2c_1|x_1|)-(|x_2|^2+2c_2|x_2|) \\
                  &< (\mu_1-1)|x_2|^2-2(\mu_1c_1|x_1|+c_2|x_2|) < 0\;,
\end{align*}
hence $H_{\vec{\lambda},\vec{c}}\cap B=\emptyset$.
The domain $B$ is a 2d cone with opening angle 45\ensuremath{^\circ},
where $\vec{v}=(\frac{1}{2}\sqrt{2-\sqrt{2}},\frac{1}{2}\sqrt{2+\sqrt{2}})$,
which proves the claim.
\item \textbf{Case $(0-)$:} $\lambda_1=0$, $\lambda_2<0$ \\
We use representation \eqref{seibold:eq:pos_stencil_function_d}.
Define $\mu_2=|\lambda_2|$. The domain $H_{\vec{\lambda},\vec{d}}$ is defined by
$g(x_1,x_2) = d_1x_1+d_2x_2-\mu_2x_2^2 \ge 0$. Due to symmetry we can assume that
$d_1,d_2\ge 0$. Two subcases have to be distinguished:
\begin{itemize}
\item \textbf{Case $d_1\neq 0$:} \\
The setup is shown in Fig.~\ref{fig_proof_pos_stencil_2d_zm1}.
For all $\vec{x}\in B$, where $B=\{(x_1,x_2)|x_1<0,x_2<0,|x_1|<|x_2|\}$, one has
$g(x_1,x_2) = -d_1|x_1|-d_2|x_2|-\mu_2x_2^2 < 0$,
hence $H_{\vec{\lambda},\vec{d}}\cap B=\emptyset$.
As above, the domain $B$ is a 2d cone with opening angle 45\ensuremath{^\circ}.
\item \textbf{Case $d_1=0$:} \\
The setup is shown in Fig.~\ref{fig_proof_pos_stencil_2d_zm2}.
The domain $g(x_1,x_2)\ge 0$ is the set $0\le x_2\le\frac{d_2}{\mu_2}$.
Any cone contained in the domain $x_2<0$ proves the claim.\vspace{-6mm}
\end{itemize}
\end{itemize}
\end{proof}
In other words, if any angle between two neighboring points (seen from the central
point) is no more than 45\ensuremath{^\circ}, then a positive stencil always exists.
\begin{rem}
The 2d cone criterion is sharp: For any $\varepsilon>0$ a point setup can be
constructed, such that all angles are less than $\frac{\pi}{4}+\varepsilon$, and a
positive stencil does not exist.
The construction is given in \cite[p.~136]{SeiboldDiss2006}. Note that the resulting
configurations are very unbalanced. Some points are very close to the central point,
others are far away. In practice, such extreme cases are typically avoided by the
construction and management of the point cloud, yielding positive stencils also for
angles significantly larger than 45\ensuremath{^\circ}.
\end{rem}

\begin{thm}[Cone criterion in 3d]
\label{seibold:thm:pos_stencil_domain3d}
Let $\vec{c},\vec{\lambda}\in\mathbb{R}^3$ with $\lambda_1+\lambda_2+\lambda_3<0$.
There exists always a cone $C_{\vec{v}}$ defined by
$\vec{v}\cdot\vec{x}>\frac{1}{\sqrt{1+\beta^2}}\|\vec{x}\|$,
where $\beta=\sqrt{\frac{1}{6}(3-\sqrt{6})}$ (a cone with total opening
angle 33.7\ensuremath{^\circ}), such that
$G_{\vec{\lambda},\vec{c}} = \mathbb{R}^d \setminus C_{\vec{v}}$ satisfies
$H_{\vec{\lambda},\vec{c}}\subset G_{\vec{\lambda},\vec{c}}$.
\end{thm}
\begin{proof}
The following cases need to be considered:
\begin{itemize}
\item \textbf{Cases $(---)$, $(0--)$, $(00-)$:} $\lambda_1\le 0$, $\lambda_2\le 0$,
$\lambda_3<0$ \\
As in 2d, we use representation
\eqref{seibold:eq:pos_stencil_function_d}. Define $\mu_i=|\lambda_i|\ \forall i=1,2,3$.
Allowing $\mu_1$ and $\mu_2$ to be zero, the domain $H_{\vec{\lambda},\vec{d}}$ is
defined by
\begin{equation*}
g(x_1,x_2,x_3) = d_1x_1+d_2x_2+d_3x_3-\mu_1x_1^2-\mu_2x_2^2-\mu_3x_3^2 \ge 0\;,
\end{equation*}
First assume that if one $\mu_i=0$, then the corresponding $d_i\neq 0$. Due to
symmetry we can w.l.o.g.~assume that $d_1,d_2,d_3\ge 0$. For all $\vec{x}\in B$,
where $B=\{(x_1,x_2,x_3)|x_1,x_2,x_3<0\}$, the function $g$ satisfies
\begin{equation*}
g(x_1,x_2,x_3) = -d_1|x_1|-d_2|x_2|-d_3|x_3|-\mu_2|x_2|^2-\mu_3|x_3|^2 < 0\;.
\end{equation*}
The domain $B$ is not a cone, but the corresponding domain from the case $(++-)$
is contained in it. Hence, the cone constructed in that case can be used here.

In the case that $\mu_i=0$ and $d_i=0$, the geometry reduces to the 2d (or trivial 1d)
case. Since in 2d the desired estimates have been shown for a larger opening angle,
the constructions transfer to the 3d case.
\item \textbf{Case $(++-)$:} $\lambda_1>0$, $\lambda_2>0$, $\lambda_3<0$ \\
Define $\mu_1 = \frac{|\lambda_1|}{|\lambda_3|}<1$ and
$\mu_2 = \frac{|\lambda_2|}{|\lambda_3|}<1$. Then $H_{\vec{\lambda},\vec{c}}$ is
defined by
\begin{equation*}
\tilde g(x_1,x_2,x_3) = \mu_1(x_1^2-2c_1x_1)
+\mu_2(x_2^2-2c_2x_2)-(x_3^2-2c_3x_3) \ge 0\;.
\end{equation*}
Due to symmetry we can assume $c_1,c_2,c_3\ge 0$. For all $\vec{x}\in B$,
where $B=\{(x_1,x_2,x_3)|x_1,x_2>0,x_3<0,|x_1|,|x_2|<\sqrt{\frac{1}{2}}|x_3|\}$,
we have
\begin{align*}
\tilde g(\vec{x})
&= \mu_1(|x_1|^2-2c_1|x_1|)+\mu_2(|x_2|^2-2c_2|x_2|)-(|x_3|^2+2c_3|x_3|) \\
&< (\tfrac{1}{2}(\mu_1+\mu_2)-1)|x_3|^2-2(\mu_1c_1|x_1|+\mu_2c_2|x_2|+c_3|x_3|) < 0\;.
\end{align*}
Note that $B$ is not a cone. However, a 3d cone can always be contained inside $B$.
Some geometric considerations yield that the cone with maximum opening angle contained
inside $B$ is given by $\beta=\sqrt{\frac{1}{6}(3-\sqrt{6})}$ and
$\vec{v}=\frac{1}{\sqrt{41-16\sqrt{6}}}\prn{2(\sqrt{3}-\sqrt{2}),2(\sqrt{3}-\sqrt{2}),1}$.
\item \textbf{Case $(+0-)$:} $\lambda_1>0$, $\lambda_2=0$, $\lambda_3<0$ \\
As in the preceding degenerate cases, we describe the domain by
\eqref{seibold:eq:pos_stencil_function_d}. Define $\mu_1=|\lambda_1|$ and
$\mu_3=|\lambda_3|$. The case $d_2=0$ reduces to the 2d case $(+-)$.
Hence, w.l.o.g.~we consider $d_1\ge 0$, $d_2>0$, $d_3\ge 0$. For all $\vec{x}\in B$,
where $B=\{(x_1,x_2,x_3)|x_1,x_2,x_3<0,|x_1|<|x_3|\}$, we have
\begin{equation*}
g(x_1,x_2,x_3) = -d_1|x_1|-d_2|x_2|-d_3|x_3|+\mu_1|x_1|^2-\mu_3|x_3|^2 < 0\;.
\end{equation*}
The estimate holds, since $\mu_1<\mu_3$ and $|x_1|<|x_3|$.
As before, a 3d cone with desired opening angle can be contained in $B$.
The cone from the case $(++-)$ can be used here.
\item \textbf{Case $(+--)$:} $\lambda_1>0$, $\lambda_2<0$, $\lambda_3<0$ \\
Define $\mu_2=\frac{|\lambda_2|}{|\lambda_1|}$ and
$\mu_3=\frac{|\lambda_3|}{|\lambda_1|}$. Since $\mu_2+\mu_3>1$, we assume
w.l.o.g.~$\mu_3\ge\frac{1}{2}$. The domain $H_{\vec{\lambda},\vec{c}}$ is defined by
\begin{equation*}
\tilde g(x_1,x_2,x_3)
= (x_1^2-2c_1x_1)-\mu_2(x_2^2-2c_2x_2)-\mu_3(x_3^2-2c_3x_3) \ge 0\;.
\end{equation*}
Due to symmetry we can assume $c_1,c_2,c_3\ge 0$. For all $\vec{x}\in B$,
where $B=\{(x_1,x_2,x_3)|x_1>0,x_2,x_3<0,|x_1|,|x_2|<\sqrt{\frac{1}{2}}|x_3|\}$
we have
\begin{align*}
\tilde g(\vec{x})
&= (|x_1|^2-2c_1|x_1|)-\mu_2(|x_2|^2+2c_2|x_2|)-\mu_3(|x_3|^2+2c_3|x_3|) \\
&= \underbrace{(|x_1|^2-\mu_2|x_2|^2-\mu_3|x_3|^2)}_{<(\frac{1}{2}-\mu_3)|x_3|^2\le 0}
-2(\mu_1c_1|x_1|+\mu_2c_2|x_2|+c_3|x_3|) < 0
\end{align*}
The domain $B$ is the same as in case $(++-)$, merely reflected at the $x_1,x_3$
plane. Hence, a 3d cone can be placed in the same way.\vspace{-5mm}
\end{itemize}
\end{proof}

\begin{rem}
Unlike the 2d case, the 3d estimate is not sharp, due to the intermediate domain $B$.
With significantly more algebra, it is possible to gain an opening angle that is a
couple of degrees larger.
\end{rem}

\begin{rem}
The existence of a positive stencil implies the existence of a stencil.
Configurations that yield an unsolvable Vandermonde system
\eqref{seibold:eq:linear_system} are automatically excluded by the cone criterion.
\end{rem}

\begin{defn}
\label{seibold:def:points_distributed_nicely}
We call points \emph{distributed nicely} around a central point, if in a test cone,
with opening angle given by Thm.~\ref{seibold:thm:pos_stencil_domain2d} respectively
Thm.~\ref{seibold:thm:pos_stencil_domain3d}, always points are contained, for any
possible direction the cone points to.
\end{defn}

%---------------------------------------------------------------------------------------------
\subsection{Condition on Point Cloud Geometry}
\label{subsec:mps_conditions_geometry}
%---------------------------------------------------------------------------------------------
The cone criterion guarantees positive stencils. We now provide conditions on the point
cloud geometry and the choice of candidate points, such that the cone criterion is
guaranteed to be satisfied. As in \cite{Levin1998}, we define
\begin{defn}
\label{seibold:def:mesh_size}
Let $\Omega\subset\mathbb{R}^d$ be a domain and $X=\{\vec{x}_1,\dots,\vec{x}_n\}$
a point cloud. The \emph{mesh size} $h$ is defined as the minimal real number, such
that $\bar\Omega\subset \bigcup_{i=1}^n \bar B\prn{\vec{x}_i,\frac{h}{2}}$,
where $\bar B\prn{\vec{x},r}$ is the closed ball of radius $r$ centered in $\vec{x}$
and $\bar\Omega$ is the closure of $\Omega$.
\end{defn}
We assume that a desired maximum mesh size is preserved by management of the point
cloud, e.g.~by inserting points into large holes.
\begin{thm}
\label{seibold:thm:pos_stencil_mesh_size}
Let the point cloud have mesh size $h$. Let $\gamma$ be the opening angle of the
cone derived in Thm.~\ref{seibold:thm:pos_stencil_domain2d} respectively
Thm.~\ref{seibold:thm:pos_stencil_domain3d}. If the radius of considered candidate
points satisfies $r > \frac{1}{\sin(\gamma/2)}\frac{h}{2}$, then for every interior
point which is sufficiently far from the boundary, a positive stencil exists.
\end{thm}
\begin{proof}
Having mesh size $h$ implies that there are no holes larger in diameter than $h$,
i.e.~$\forall\vec{x}\in\Omega\;\exists\vec{x}_i\in X:\|\vec{x}_i-\vec{x}\|<\frac{h}{2}$.
Fig.~\ref{seibold:fig_cone_meshsize} shows a ball with radius $r$ around the central
point and a cone with opening angle $\gamma$. If the cone contains no point, there
must be a ball of radius $\frac{h}{2}$ which contains no points. The claim follows by
considering the triangle ($0,\vec{x},\vec{x}_i$).
\end{proof}
The specific ratios of candidate radius to maximum hole size radius are
\begin{equation*}
\frac{r}{h/2} >
\sqrt{1+\tfrac{1}{\beta^2}} =
\begin{cases}
\sqrt{4+2\sqrt{2}} &= 2.61 \quad\mathrm{in~2d} \\
\sqrt{7+2\sqrt{6}} &= 3.45 \quad\mathrm{in~3d}
\end{cases}
\end{equation*}
Using sharper estimates, the 3d ratio can be lowered to $\sqrt{6+2\sqrt{6}} = 3.30$.
In practice, point clouds are much nicer than the worst case scenario, so significantly
smaller ratios lead to positive stencils.

\begin{figure}
\centering
\begin{minipage}[t]{.48\textwidth}
\centering
\includegraphics[width=0.9\textwidth]{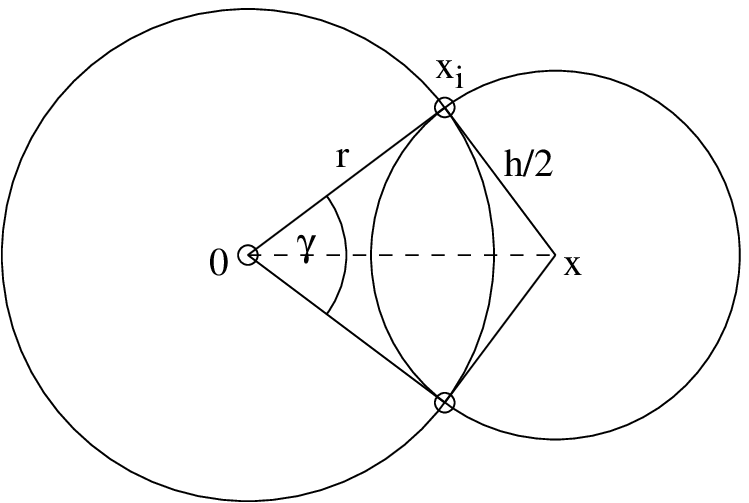}
\caption{Relation between neighborhood radius and mesh size}
\label{seibold:fig_cone_meshsize}
\end{minipage}
\hfill
\begin{minipage}[t]{.48\textwidth}
\centering
\includegraphics[width=0.9\textwidth]{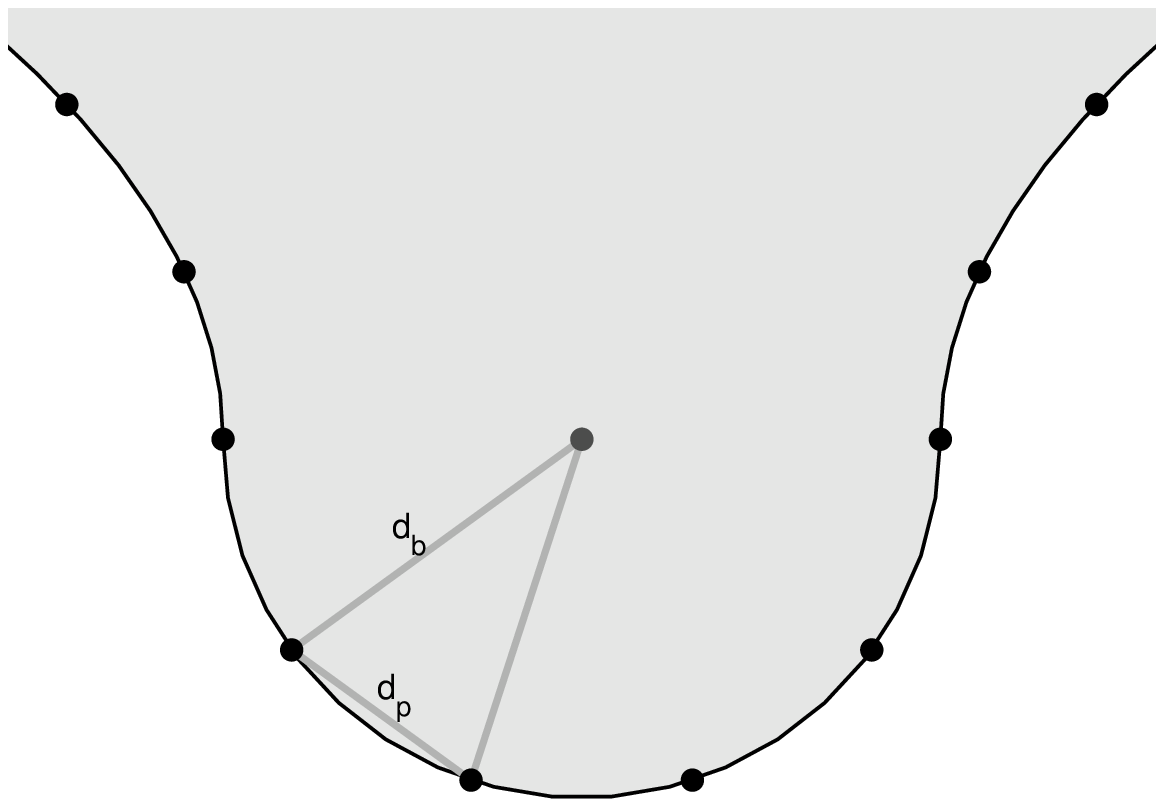}
\caption{Guaranteeing the cone criterion close to the boundary}
\label{seibold:fig_cone_boundary}
\end{minipage}
\end{figure}

Thm.~\ref{seibold:thm:pos_stencil_mesh_size} is valid for any interior point which
is far enough from the boundary, that the mesh size criterion guarantees points to
lie between the point in consideration and the boundary. For a layer of interior
points close to the boundary, the cone criterion can be enforced by the following
construction (see Fig.~\ref{seibold:fig_cone_boundary}):
First, place boundary points sufficiently dense. Let their maximum distance be
$d_\mathrm{p}$. Second, ensure that every interior point has a minimum distance
$d_\mathrm{b}$ from the boundary. In 2d, this is
$d_\mathrm{b}>\frac{4}{\pi}d_\mathrm{p}$.

%---------------------------------------------------------------------------------------------
\subsection{Neumann Boundary Points}
%---------------------------------------------------------------------------------------------
Assume the boundary $\partial\Omega$ is $C^1$ around Neumann boundary points.
Consider a local coordinate system, i.e.~$\vec{n}=(1,0)$ in 2d, respectively
$\vec{n}=(1,0,0)$ in 3d. We obtain Neumann stencils by solving the linear
minimization problem
\begin{equation}
\min \sum_{i=1}^m \frac{s_i}{w_i}, \ \mathrm{s.t.} \
V\cdot\vec{s}=\vec{n}, \ \vec{s}\ge 0\;,
\label{seibold:eq:mps_neumann_system}
\end{equation}
where the matrix $V$ is given by \eqref{seibold:eq:constraints_Neumann} as
\begin{equation*}
V = \prn{\begin{array}{ccc}
x_1 & \dots & x_m \\
y_1 & \dots & y_m
\end{array}}
\mbox{~in 2d, and~}
V = \prn{\begin{array}{ccc}
x_1 & \dots & x_m \\
y_1 & \dots & y_m \\
z_1 & \dots & z_m
\end{array}}
\mbox{~in 3d.}
\end{equation*}
We consider the 3d case. The 2d geometry is contained as a special case. For an
easier analysis, we consider a locally convex domain, i.e.~$x_i\ge 0\ \forall i$.
\begin{thm}
For a Neumann boundary point a positive stencil exists, iff the points' projections
onto the normal plane do not lie all in one and the same half space.
\end{thm}
\begin{proof}
Due to Farkas' lemma, system \eqref{seibold:eq:mps_neumann_system} has no solution
$\vec{s}\ge 0$, iff the system $V^T\cdot\vec{w}\ge 0$ has a solution satisfying
$w_x<0$, where $\vec{w}=(w_x,w_y,w_z)^T$.

Let no positive stencil exist. Then $\vec{w}\in\mathbb{R}^d$ with $w_x<0$ exists,
such that $V^T\cdot\vec{w}\ge 0$, i.e.~$w_xx_i+w_yy_i+w_zz_i\ge 0 \ \forall i$.
This is equivalent to 
$\vec{k}\cdot\prn{\begin{array}{c}y_i\\z_i\end{array}}\ge x_i \ \forall i$,
where $\vec{k}=(\frac{w_y}{|w_x|},\frac{w_z}{|w_x|})$.
This means that the $y$-$z$ projection of all points lies in one and the same half
space (in the direction of $\vec{k}$).

Conversely, assume that the $y$-$z$ projection of all points lies in one and the
same half space. Let $I$ be the indices of all points in consideration. Define
$I_p=\{i\in I: \ x_i>0\}$. Consider w.l.o.g.~the case $z_i\ge 0 \ \forall i$, where
$z_i>0 \ \forall i\in I_p$. Choose
$\vec{w}=\prn{-1,0,\frac{\max_{i\in I} x_i}{\min_{i\in I_p} z_i}}$.
Then for all $i\in I_p$ it holds
\begin{equation*}
\vec{w}^T\cdot\vec{x}_i = -x_i+\frac{\max_{i\in I} x_i}{\min_{i\in I_p} z_i}z_i
\ge -x_i+\max_{i\in I} x_i\ge 0\;,
\end{equation*}
and for all $i\in I\setminus I_p$ one has $\vec{w}^T\cdot\vec{x}_i\ge 0$,
since $x_i=0$. Thus, no positive stencil exists.
\end{proof}
\begin{rem}
\label{seibold:rem:Neumann_first_order}
Construction \eqref{seibold:eq:mps_neumann_system} yields a first order accurate
approximation of the normal derivative. Second order accuracy could be achieved,
by including quadratic terms into $V$. However, in this case no positive stencil
exists, since the condition $\sum_{i=1}^m (x_i^2+y_i^2)s_i = 0$ cannot be satisfied.
\end{rem}

%---------------------------------------------------------------------------------------------
\subsection{Treatment of Cracks}
%---------------------------------------------------------------------------------------------
A non-convex part of the domain $\Omega$ requires a special treatment if it is thinner
than the local neighborhood radius. Fig.~\ref{seibold:fig_crack} shows a proper
treatment of a crack. Stencils on one side of the crack must not use points on the
other side. In the MPS method this property can be guaranteed by the following
construction: For a central point $\vec{x}_0\in\Omega$, only circular neighbors inside
the star shaped core
$\bar{\Omega}_{\vec{x}_0} =\{\vec{x}\in\bar{\Omega}:
(1-\alpha)\vec{x}_0+\alpha\vec{x}\in\bar{\Omega}\ \forall\alpha\in [0,1]\}$
(bold dots in Fig.~\ref{seibold:fig_crack}) are considered as candidates for the
linear minimization \eqref{seibold:eq:linear_minimization}.
If the domain is defined implicitly $\Omega = \{\vec{x}:\phi(\vec{x})<0\}$,
the point $\vec{x}$ does not lie in $\bar{\Omega}_{\vec{x}_0}$ if a point
$\vec{y}\in [\vec{x}_0,\vec{x}]$ on the connection line satisfies $\phi(\vec{y})>0$.

\begin{figure}
\centering
\begin{minipage}[t]{.39\textwidth}
\centering
\includegraphics[width=0.99\textwidth]{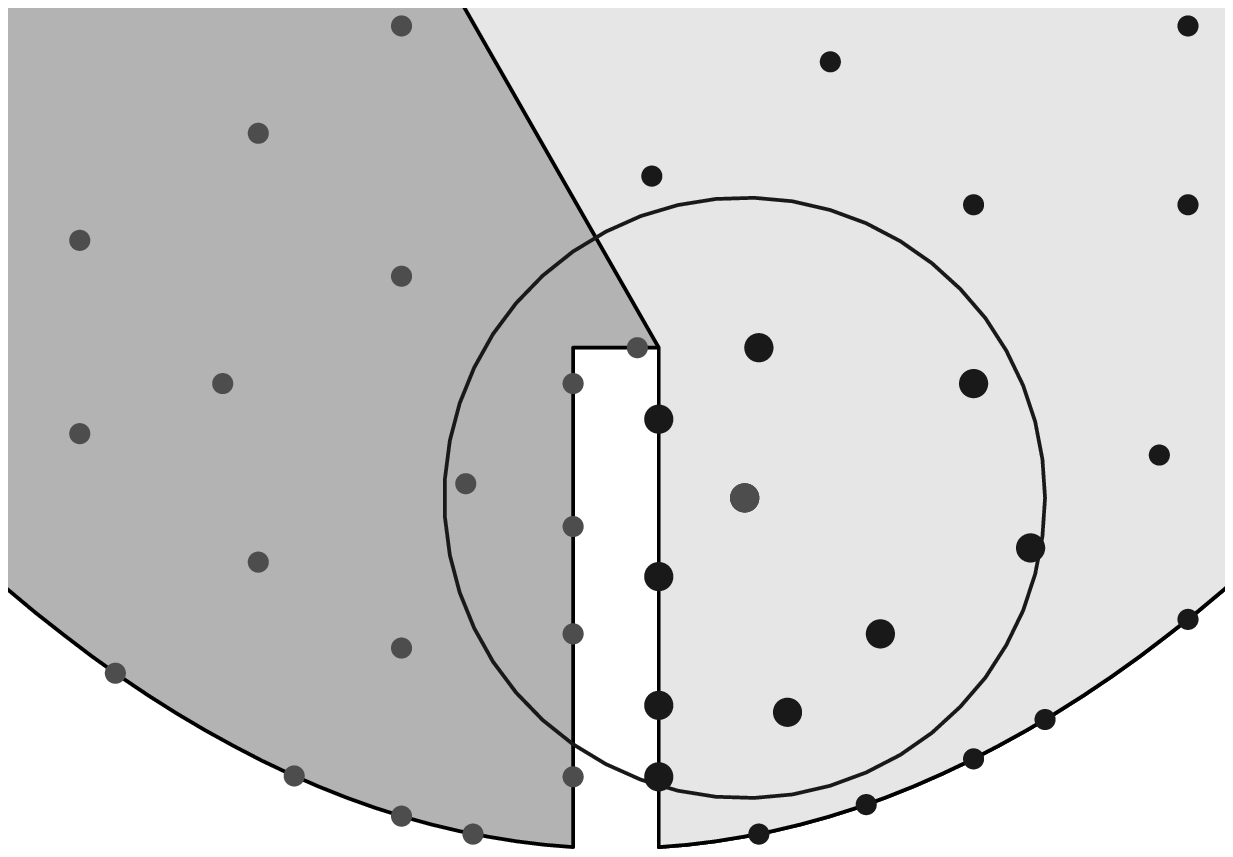}
\caption{Proper treatment of a crack}
\label{seibold:fig_crack}
\end{minipage}
\hfill
\begin{minipage}[t]{.58\textwidth}
\centering
\includegraphics[width=0.5\textwidth]{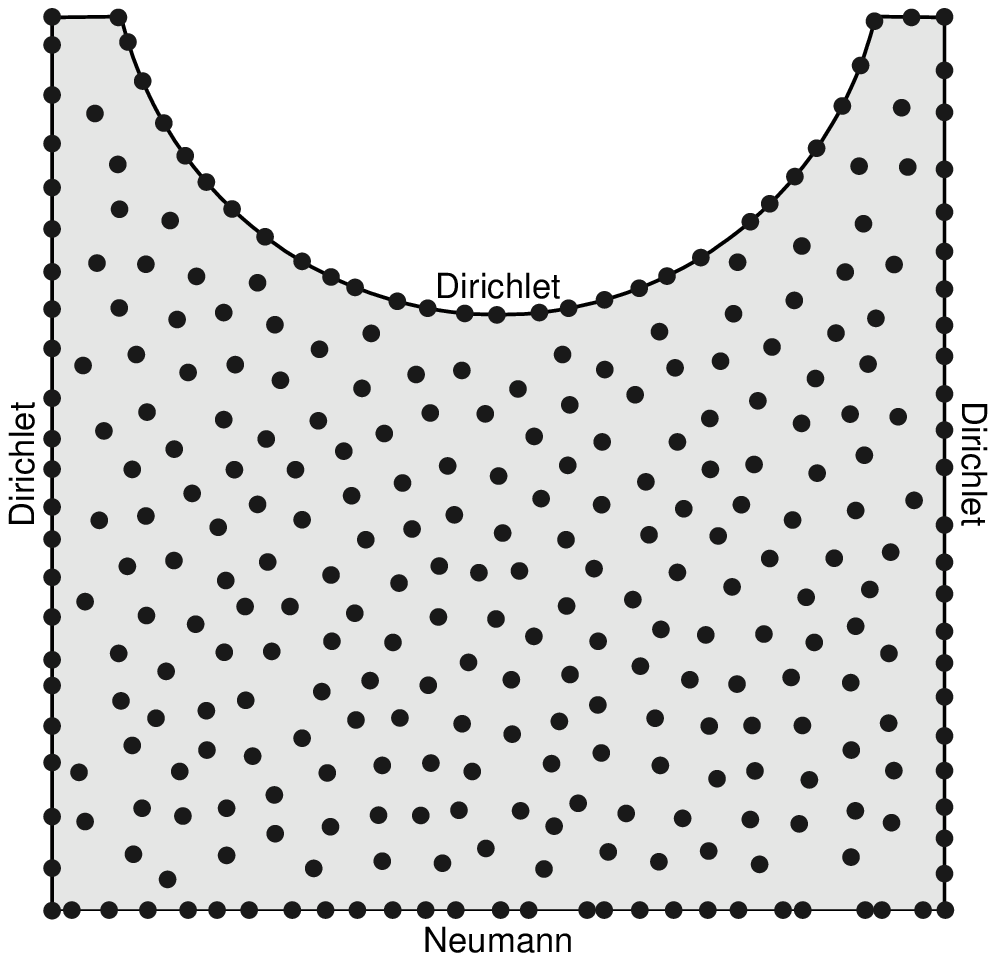}
\hfill
\includegraphics[width=0.42\textwidth]{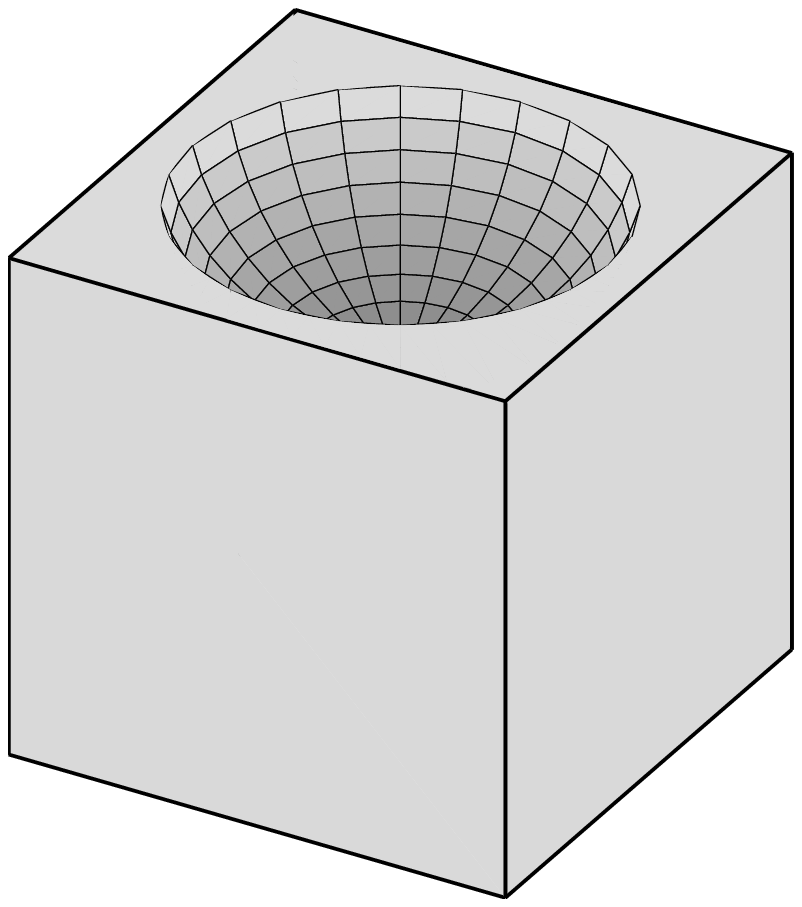}
\caption{Computational domain in 2d and 3d}
\label{seibold:fig_domain}
\end{minipage}
\end{figure}

%=============================================================================================
\section{Minimal Stencils and Matrix Connectivity}
\label{seibold:sec:lp_matrix_connectivity}
%=============================================================================================
Due to Thm.~\ref{seibold:thm:ess_irred_is_ess_diagdom} and
Thm.~\ref{seibold:thm:Lmatrix_is_Mmatrix}, the matrix composed of positive stencils
is an M-matrix, if every interior and Neumann boundary point is connected to a
Dirichlet boundary point.
\begin{thm}
\label{seibold:thm:mps_reach_boundary}
Consider the Poisson problem \eqref{seibold:eq:poisson_equation} on a domain which has
no holes (i.e.~only an outer boundary). With a MPS discretization every interior point
is connected to a boundary point.
\end{thm}
\begin{proof}
Assume there is a point $i\in I$ which is not connected to a boundary point.
Define $I_i=\{j\in I:\mbox{$i$ is connected to $j$}\}$. Every point in $I_i$ is not
connected to a boundary point. Hence, the set $I_i\subset I$ forms an island
inside $\Omega$ which does not reach a the boundary. Consider a point that spans the
convex hull of $I_i$. It only uses points in its stencil that lie inside the island,
hence these lie in one and the same half space, which contradicts the necessary
condition on positive stencils given by Thm.~\ref{seibold:thm:mps_pos_stencil_necessary}.
\end{proof}
\begin{rem}
Although Thm.~\ref{seibold:thm:mps_reach_boundary} does not extend to interior
boundaries, in practice the MPS works for these, given enough boundary points are
placed.
\end{rem}
It remains to ensure that every point connects to a \emph{Dirichlet} boundary point.
Unfortunately, this cannot be concluded from the MPS method directly. It is possible
that an isolated Dirichlet point is not used in the stencils of nearby interior points. 
Note that this phenomenon can also happen on regular grids. A single Dirichlet point
in a corner of a domain may not be used by regular five-point stencils.
If Dirichlet data is prescribed only in small regions, these regions have to be
equipped with a sufficient number of boundary points. In addition, the MPS
implementation has to ensure that these Dirichlet points are used by nearby points.
If done so, the MPS method guarantees to generate M-matrices.

\begin{figure}
\centering
\begin{minipage}[t]{.49\textwidth}
\centering
\includegraphics[width=0.99\textwidth]{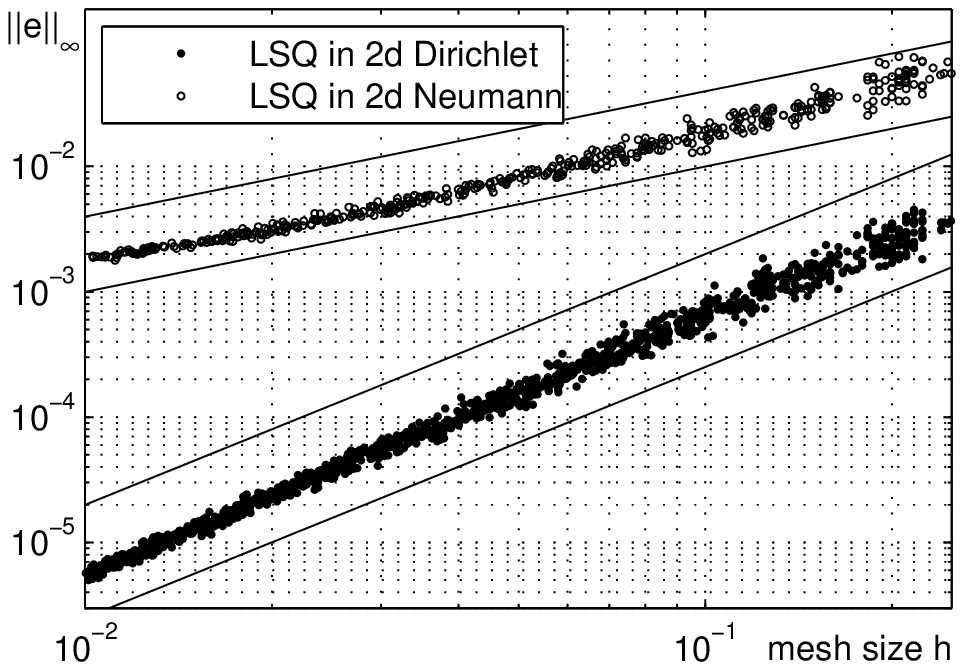}
\end{minipage}
\hfill
\begin{minipage}[t]{.49\textwidth}
\centering
\includegraphics[width=0.99\textwidth]{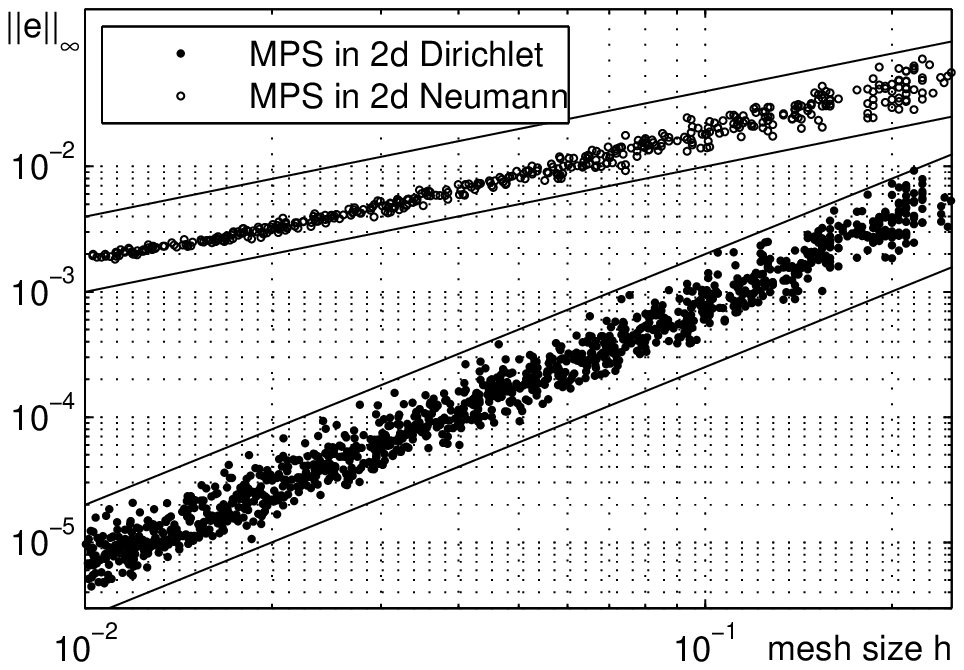}
\end{minipage}
\caption{Error convergence in 2d -- LSQ vs.~MPS method}
\label{seibold:fig_error_analysis_2d}
\end{figure}

%=============================================================================================
\section{Numerical Experiments}
\label{seibold:sec:numerics}
%=============================================================================================
We investigate the numerical accuracy of the MPS method in comparison to a least
squares approach. As test problems we consider the Poisson equation
\eqref{seibold:eq:poisson_equation} in the unit box with a ball cut out
$\Omega=[0,1]^d \setminus B(\prn{\frac{1}{2},\dots,\frac{1}{2},1.1},0.44)$.
Fig.~\ref{seibold:fig_domain} shows the computational domain in 2d and 3d.
In one case, the boundary conditions are Dirichlet everywhere, in the other case,
Neumann at the bottom $x_d = 0$, and Dirichlet everywhere else. Given $g$, we set
$f=\Delta g$ and $h = \pd{g}{n}$, so \eqref{seibold:eq:poisson_equation} has the
solution $u=g$. Specifically, we choose
$g(x_1,x_2) = \tfrac{1}{c_2}\prn{x_1\sin(x_2+2)+x_2\sin(2x_1+1)}$ in 2d and
$g(x_1,x_2,x_3) = \tfrac{1}{c_3}\prn{x_1\sin(x_2+2)+x_2\sin(2x_3+3)+x_3\sin(3x_1+1)}$
in 3d, with $c_2$ and $c_3$ such that $\max g-\min g=1$.
The problem is discretized by a sequence of point clouds. The point clouds have a
uniform average density and a minimum separation \cite{Levin1998} of $\delta=0.05$.
Each point cloud is managed to satisfy the conditions for the existence of positive
stencils, as derived in Sect.~\ref{subsec:mps_conditions_geometry}. Since to one mesh
size $h$, given by Def.~\ref{seibold:def:mesh_size}, many point clouds exist, we
sample a number of experiments to obtain an average error convergence rate.
To every point cloud we apply a weighted least squares method
(Sect.~\ref{seibold:sec:least_squares_approaches}) and the MPS methods
(Sect.~\ref{seibold:sec:lp_approach}), both with $w(\delta)=\delta^{-4}$.

\begin{figure}
\centering
\begin{minipage}[t]{.49\textwidth}
\centering
\includegraphics[width=0.99\textwidth]{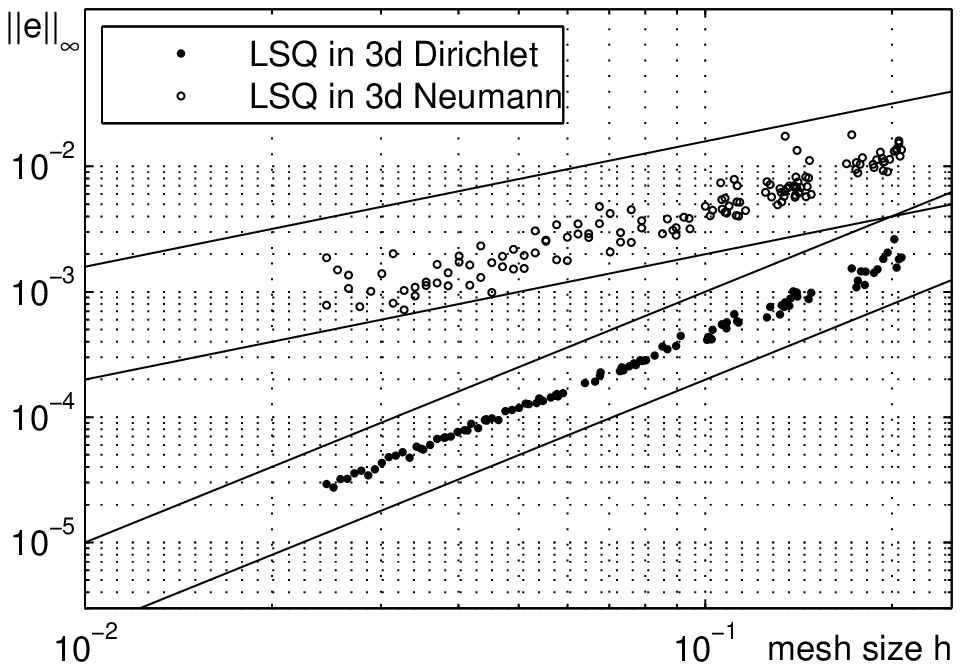}
\end{minipage}
\hfill
\begin{minipage}[t]{.49\textwidth}
\centering
\includegraphics[width=0.99\textwidth]{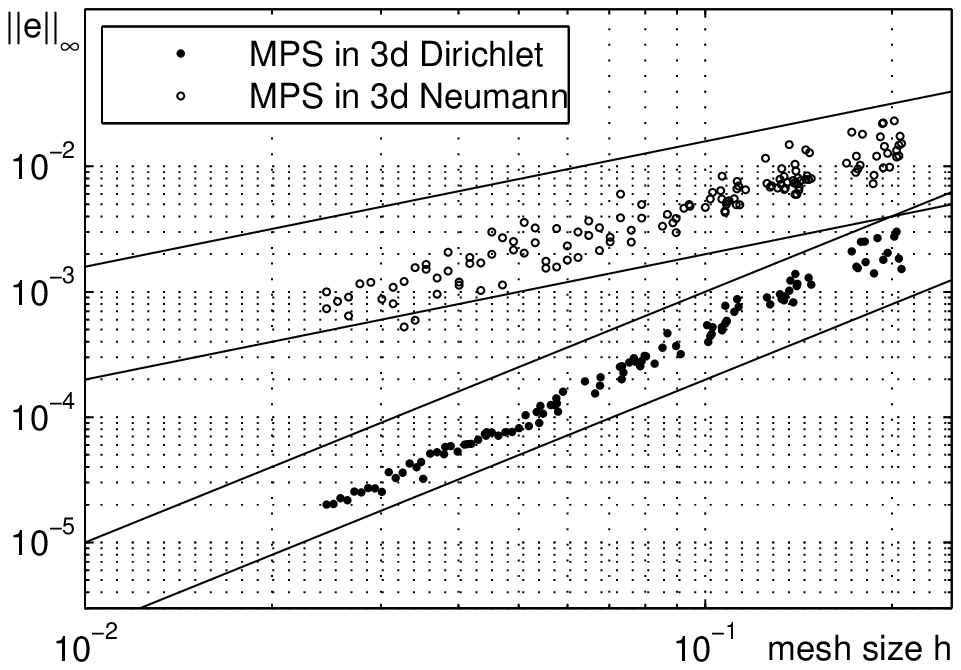}
\end{minipage}
\caption{Error convergence in 3d -- LSQ vs.~MPS method}
\label{seibold:fig_error_analysis_3d}
\end{figure}

The numerical results are shown in
Fig.~\ref{seibold:fig_error_analysis_2d} for 2d, and
Fig.~\ref{seibold:fig_error_analysis_3d} for 3d.
Plotted is the error measured in the maximum norm over the mesh size $h$.
Solid dots represent the all Dirichlet version of a problem, while open circles show
the error with partial Neumann boundary conditions. The reference lines are of slope
one and two respectively. One can observe the following:
\begin{itemize}
\item
Both approaches show a second order convergence rate for the pure Dirichlet problem,
and first order convergence if Neumann boundary conditions are involved.
While the derivation in Sect.~\ref{seibold:subsec:consistent_derivatives} enforces
only a first order accurate approximation of the Laplacian at interior points,
point clouds tend to possess enough averaged symmetry to actually yield second order
error convergence. On the other hand, the first order accurate approximation
(Rem.~\ref{seibold:rem:Neumann_first_order}) at Neumann boundary points carries
through.
\item
The MPS method shows a larger variation in error over the ensemble of
experiments. One reason for this effect could be the discontinuous dependence on the
point positions (see Rem.~\ref{seibold:rem:no_continuous_dependence}).
For two similar point clouds, the MPS method may select very different stencils.
\item
Both methods yield roughly the same error constant. The average error is slightly
lower with the MPS method.
\end{itemize}

%---------------------------------------------------------------------------------------------
\subsection{Computational Cost}
%---------------------------------------------------------------------------------------------
For a 3d test problem, the MPS method is compared with the LSQ method in terms of
computational cost. For a sequence of point clouds, Fig.~\ref{seibold:fig_cputimes}
shows the CPU times for the setup of the system matrix (left plot), the solution
of the arising system with a BiCG scheme (center plot), and the solution with an
algebraic multigrid (AMG) scheme (right plot). The latter is performed using
SAMG \cite{SAMG}
by the \emph{Fraunhofer Institute for Algorithms and Scientific Computing}.

\begin{figure}
\centering
\begin{minipage}[t]{.32\textwidth}
\centering
\includegraphics[width=0.99\textwidth]{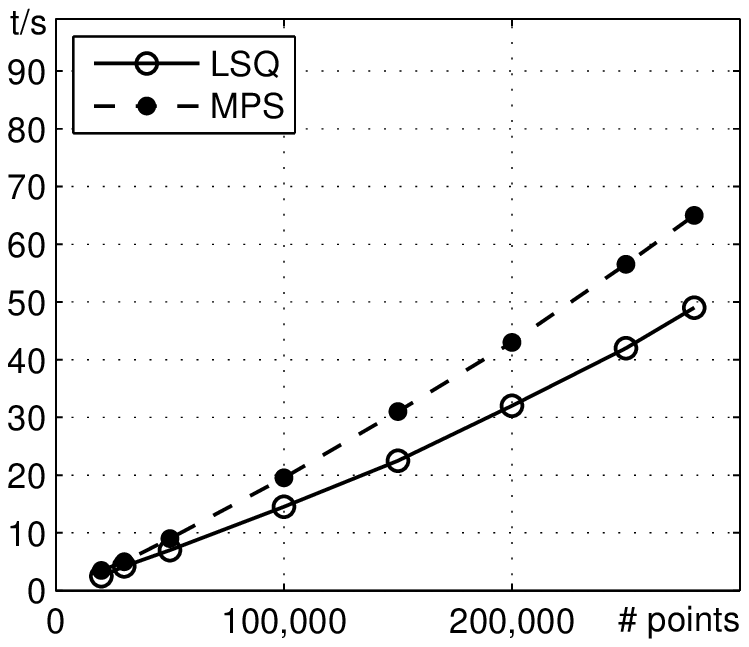}
\end{minipage}
\hfill
\begin{minipage}[t]{.32\textwidth}
\centering
\includegraphics[width=0.99\textwidth]{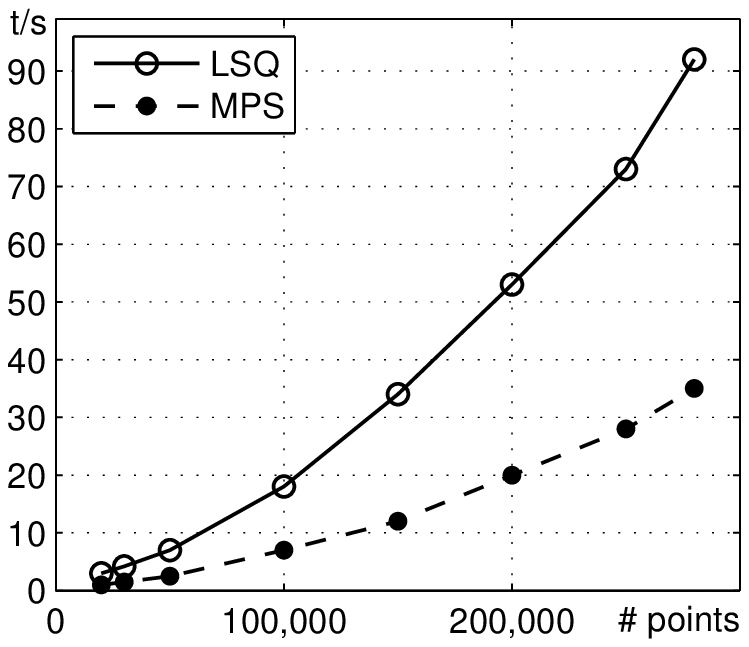}
\end{minipage}
\hfill
\begin{minipage}[t]{.32\textwidth}
\centering
\includegraphics[width=0.99\textwidth]{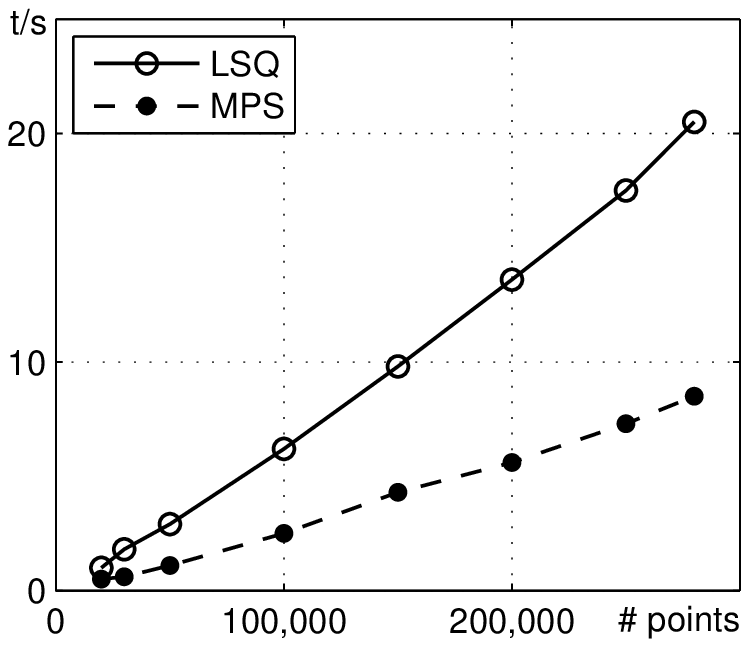}
\end{minipage}
\caption{CPU times for setup and solve by BiCG and AMG}
\label{seibold:fig_cputimes}
\end{figure}

As expected (Sect.~\ref{subsec:solving_lp}), the cost of setting up the system matrix
is roughly equal for MPS and LSQ method. In fact, MPS is slightly slower with the used
simplex method. However, more efficient linear programming methods may turn the tide
towards the MPS method. On the other hand, the cost for solving the large linear system
is significantly reduced by the MPS method. The speedup factor equals the factor in
sparsity, i.e.~the MPS approximation does not modify the convergence rate. While the
AMG solver shows a cost roughly linear in the number of unknowns, solvers further away
from optimal effort (like BiCG) will greatly benefit from the sparsity of the MPS
approximation as the number of unknowns increases.

%=============================================================================================
\section{Conclusions and Outlook}
%=============================================================================================
We have presented a meshfree approach that constructs minimal positive stencils for the
Laplace operator on a cloud of points. We have shown that under moderate assumptions
on the local resolution of the point cloud, positive stencils always exist.
The method approximates the Poisson equation by M-matrices, which are optimally sparse.
Both properties are beneficial, and they are in general not met by classical least
squares approaches. Numerical tests show that with the presented approach, the
approximation error roughly equals the error of classical approaches. The computational
cost to construct the new approximation is comparable to the cost of least squares
methods. On the other hand, for solving the arising linear system, both cost and memory
requirements are reduced significantly due to the optimal sparsity.

An efficient solution of the linear programs is a crucial point in the presented method,
worth a deeper analysis. The application to particle methods shall be investigated.
While the presented approach can stand as a method of its own, it may also increase the
efficiency of other approaches. The minimal stencils can be augmented by additional
neighbors and the final stencil be computed by a least squares method. For instance, a
local neighborhood radius can be based on the farthest minimal stencil point, thus
increasing sparsity and adding local adaptivity to existing meshfree codes.

%=============================================================================================
\section*{Acknowledgments}
%=============================================================================================
We would like to thank
Axel Klar, J{\"o}rg Kuhnert, Sven Krumke, Helmut Neunzert, and Sudarshan Tiwari
for helpful discussions and comments.
The support by the National Science Foundation is acknowledged.
The author was partially supported by NSF grant DMS--0813648.

%=============================================================================================
\bibliographystyle{plain}
\bibliography{meshfree_poisson.bib}
%=============================================================================================

\end{document}